\documentclass[a4paper,11pt]{article}
\usepackage{etex}
\pdfoutput=1

\usepackage{fullpage} 
\usepackage{lmodern}
\usepackage[protrusion=true,expansion=true]{microtype}

\usepackage{enumerate}

\usepackage[utf8]{inputenc}
\usepackage[english]{babel}
\usepackage[T1]{fontenc}
\usepackage{amssymb,amsmath,amsthm}
\usepackage{url}
\usepackage{graphicx}
\usepackage{mathtools}

\usepackage{pst-all}

\usepackage{tikz-cd}
\usepackage{amscd} 
\usepackage{bbm} 
\usepackage{accents}
\usepackage{todonotes}
\usepackage{longtable}
\usepackage{tabulary}

\newcommand{\RR}{\mathbb{R}}
\newcommand{\CC}{\mathbb{C}}
\newcommand{\ZZ}{\mathbb{Z}}
\newcommand{\NN}{\mathbb{N}}
\newcommand{\QQ}{\mathbb{Q}}

\newcommand{\mcA}{\mathcal A}
\newcommand{\mcB}{\mathcal B}
\newcommand{\mcC}{\mathcal C}
\newcommand{\mcE}{\mathcal E}

\newcommand{\mcJ}{\mathcal J}
\newcommand{\mcL}{\mathcal L}
\newcommand{\mcM}{\mathcal M}
\newcommand{\mcN}{\mathcal N}

\newcommand{\defeq}{\vcentcolon=}

\newcommand{\id}{\operatorname{id}}

\newcommand{\rank}{\operatorname{rank}}

\newcommand{\ind}{\operatorname{ind}}
\newcommand{\lcm}{\operatorname{lcm}}

\newcommand{\im}{\operatorname{im}}

\newcommand{\dd}{\partial}

\newcommand{\ev}{\mathrm{ev}}

\newcommand{\RS}{\mathrm{RS}}
\newcommand{\CZ}{\mathrm{CZ}}
\newcommand{\symm}{\mathrm{symm}}
\newcommand{\reg}{\mathrm{reg}}
\newcommand{\fixed}{\mathrm{fixed}}

\DeclarePairedDelimiter{\ceil}{\lceil}{\rceil}
\DeclarePairedDelimiter{\floor}{\lfloor}{\rfloor}

\newtheorem{thm}{Theorem}[section]
\newtheorem{prop}[thm]{Proposition}
\newtheorem{cor}[thm]{Corollary}
\newtheorem{lem}[thm]{Lemma}

\theoremstyle{definition}
\newtheorem{defn}{Definition}[section]

\newtheorem*{rmk}{Remark}

\title{Symplectic homology of some Brieskorn manifolds}
\author{Peter Uebele}
\date{}

\begin{document}

\maketitle

\begin{abstract}
 This paper consists of two parts. In the first part, we use symplectic homology to distinguish the contact structures on the Brieskorn manifolds $\Sigma(2\ell,2,2,2)$, which contact homology cannot distinguish. This answers a question from \cite{KvK}.
 
 In the second part, we prove the existence of infinitely many exotic but homotopically trivial exotic contact structures on $S^7$, distinguished by the mean Euler characteristic of $S^1$-equivariant symplectic homology. Apart from various connected sum constructions, these contact structures can be taken from the Brieskorn manifolds $\Sigma(78k+1,13,6,3,3)$. We end with some considerations about extending this result to higher dimensions.
\end{abstract}

\section{Introduction}

Brieskorn manifolds, defined as
\[
 \Sigma(a_0,\ldots,a_n)\defeq \{z\in \CC^{n+1} \mid z_0^{a_0} + \cdots + z_n^{a_n} = 0, \|z\|=1 \}
\]
for integers $a_i\geq 2$, have been a rich source of interesting examples in geometry and topology, for instance for the discovery of exotic spheres.

In contact topology, they became prominent after Ustilovsky \cite{Ustil} used them to show that the spheres $S^{4m+1}$ carry infinitely many different contact structures in each formal homotopy class. To prove this, he used certain Brieskorn manifolds that are diffeomorphic to spheres, but whose canonical contact structures (see Section~\ref{sec_def_Brieskorn}) have different contact homology. Later on, van Koert \cite{vanKoert} calculated the cylindrical contact homology for all Brieskorn manifolds for which it is (conjecturally) well-defined, using Morse--Bott methods from \cite{Bourgeois_thesis}. In another event, Fauck \cite{Fauck} has reproven Ustilovsky's theorem using Rabinowitz Floer homology, which has the advantage that its analytic foundations are well-established.

In this text, we will examine the manifolds $\Sigma(2\ell,2,2,2)$ for $\ell\geq 1$, which are all diffeomorphic to $S^2\times S^3$ (i.e.\ the unit cotangent bundle of $S^3$). As was pointed out in \cite{vanKoert}, these manifolds have the same contact homology. The same applies to positive $S^1$-equivariant symplectic homology, see \cite{KvK}. Moreover, the underlying almost contact structures coincide, as follows from \cite[Proposition~8.1.1]{Geiges-book} and the fact that their first Chern class vanishes. Thus, the question whether they (or some of them) are contactomorphic was left open. This text answers this question negatively:

\begin{thm} \label{main_thm}
 The manifolds $\Sigma(2\ell,2,2,2), \ell\geq 1$ with their canonical contact structures are pairwise non-contactomorphic. Hence, there are infinitely many different contact structures on $S^2\times S^3$.
\end{thm}

At this point, we should mention two other results, by Lerman \cite{Lerman} and Abreu and Macarini \cite{Ab-Mac}, respectively, who also find infinitely many contact structures on $S^2\times S^3$. However, their examples do not overlap with the ones discussed here for the following reasons: The examples of \cite{Lerman} have non-vanishing first Chern class, whereas all Brieskorn manifolds have vanishing first Chern class. The examples of \cite{Ab-Mac} also have vanishing first Chern class, but they can be distinguished from ours by their contact homology. Namely, they all have contact homology in degree $0$, whereas the contact homology of $\Sigma(2\ell,2,2,2)$ starts in degree $2$ (see \cite[Example~3.1.1]{vanKoert}).

We prove Theorem~\ref{main_thm} by computing the (positive) symplectic homology of a symplectic filling of $\Sigma(2\ell,2,2,2)$ with $\ZZ_2$-coefficients. If the contact manifold satisfies certain index conditions (see Lemma~\ref{lem_indep_filling}), the positive symplectic homology is independent of the filling, and thus distinguishes the contact structures. 
Along the way, we also compute the positive symplectic homology of $\Sigma(2\ell,\underbrace{2,\ldots,2}_n)$ for $n\geq 5$ odd, which turns out to be much easier than for $n=3$.

The result, together with \cite{Ustil} or \cite{Fauck}, can also be viewed as a classification of the links of $A_k$-singularities as contact manifolds. These links can be defined as the Brieskorn manifolds $\Sigma(k+1,\underbrace{2,\ldots,2}_{n})$ with $k\geq 1$ and $n\geq 2$. If $n$ is even, these manifolds are already distinguished by their singular homology, because $H_{n-1}(\Sigma(k+1,2,\ldots,2)) = \ZZ_{k+1}$ in this case. For $n$ odd and $k$ even, the contact structures are distinguished in \cite{Ustil} and \cite{Fauck}. Note that their results can also be proven using symplectic homology, with computations almost identical to \cite{Fauck}. This leaves the case $n$ odd and $k$ odd, which is treated here.

Going back to Ustilovsky's result, one might wonder whether a similar statement about exotic contact structures on spheres also holds true for $S^{4m+3}$. These dimensions turn out to be more complicated, mainly because, unlike in dimensions $4m+1$, there are infinitely many formal homotopy classes of almost contact structures. Hence, it is more difficult to find contact structures representing a given formal homotopy class, e.g.\ the standard one.

Partial results in this direction were proven in \cite{Geiges-App}, \cite{Ding-Geiges} and \cite{McLean}. In particular, \cite{Ding-Geiges} shows the existence of one exotic but homotopically trivial contact structure on $S^{4m+3}$ for every $m\geq1$, while \cite[Corollary~1.5]{McLean} implies existence of at least two such contact structures on spheres of dimension $2n-1\geq 15$.

In this text, we treat mainly dimension $7$. We can show that there are in fact infinitely many exotic but homotopically trivial contact structures on $S^7$. Our method is somewhat similar to \cite{Ustil}: We use a class of Brieskorn manifolds, namely $\Sigma(78k+1,13,6,3,3)$, which we show to be all diffeomorphic to $S^7$. Moreover, their canonical contact structures all lie in the standard formal homotopy class of $S^7$. Note that \cite[Proposition~19]{Geiges-App} actually claims that such an example cannot exist. However, its proof contains a mistake, originating from different conventions about Bernoulli numbers.

In order to distinguish the contact structures, it would be very difficult to compute any variant of contact homology or symplectic homology, because there are generators in a wide range of degrees and the differential is hard to compute. Also, most of these homology theories might actually depend on the filling, in which case one cannot use them to distinguish contact structures. Instead, we will use the \emph{mean Euler characteristic}, see Section~\ref{sec_mean_Euler} for the definition. Although this quantity is derived from positive $S^1$-equivariant symplectic homology, it is in fact, whenever it is defined, an invariant of the contact structure (see \cite[Corollary~2.2]{FSvK}). Furthermore, it is much easier to compute because, much like the usual Euler characteristic of singular homology, one can do the computations on the chain level, without knowing the differential.

As for higher dimensions, it seems difficult to get a similar example. By using connected sums, one can simplify a little, thus making a result possible for dimensions $11$ and $15$. In principle, the construction might be possible for any dimension, but the computations get increasingly difficult.

\begin{thm} \label{thm_intro_exotic_spheres}
 There exist infinitely many exotic but homotopically standard contact structures on $S^7$, $S^{11}$ and $S^{15}$.
\end{thm}

\medskip

This text is organized as follows. Section~\ref{sec_Brieskorn} recalls some general facts about Brieskorn manifolds, their contact structures and their Reeb dynamics as well as their topology.
In Section~\ref{sec_MB}, we start computing the symplectic homology in a Morse--Bott framework. Merely computing the Conley--Zehnder indices turns out to be sufficient to compute the positive symplectic homology for $n\geq 5$ (see Theorem~\ref{thm_ngeq5}). However, for $n=3$, the indices lie too close to each other, so we cannot exclude differentials between the corresponding generators. Nevertheless, the Morse--Bott formalism is still useful to get upper bounds on the symplectic Betti numbers.

In the rest of Section~\ref{sec_sympl_homol}, the case $n=3$ is examined in detail. The main idea is to use a symmetry $\psi$ of $\Sigma(2\ell,2,2,2)$ to see that the Floer cylinders between certain orbits must come in pairs. Hence the differential vanishes for $\ZZ_2$-coefficients. 
To do this, we abandon the (full) Morse--Bott picture and use a perturbation, analogously to the one used by Ustilovsky in \cite{Ustil}. We first prove that there are no Floer cylinders in the fixed point set of $\psi$ (Lemma~\ref{lem_fixed}).
For a Floer cylinder $u$, we want to make sure that $\psi\circ u$ is again a Floer cylinder (Proposition~\ref{prop_reparametrization}). For this, we need to use an almost complex structure $J$ which is symmetric under $\psi$. This, in turn, requires a new proof of the transversality of the relevant moduli spaces, which we do in Proposition~\ref{prop_transversality}. Putting these ingredients together gives the symplectic homology of $\Sigma(2\ell,2,2,2)$.

Section~\ref{sec_exotic_S7} deals with the question about exotic contact structures on spheres, with the main example being $\Sigma(78k+1,13,6,3,3)$. By the result of Sections \ref{sec_diffeo_type} and \ref{sec_ac}, it has the standard smooth and almost contact structures, respectively. In Section~\ref{sec_mean_Euler}, the different values of $k$ are distinguished by the mean Euler characteristic. Finally, Section~\ref{sec_how_to_find} explains the algorithm how this example was found. We end with some remarks about generalizations to different almost contact structures, exotic $7$-spheres and higher dimensions in Section~\ref{sec_further_results}.

\vspace{0.3cm}
\noindent 
\emph{Acknowledgements.} First and foremost, I would like to thank my advisor Kai Cieliebak for introducing me to this subject, as well as for his continued support and advice. I am also grateful to Peter Albers for an important correction, and to Alexander Fauck, Otto van Koert and my colleague Sven Prüfer for various helpful discussions. Furthermore, I want to thank the anonymous referee for the careful reading and several helpful comments.

\section{Brieskorn manifolds} \label{sec_Brieskorn}

\subsection{Basic definitions} \label{sec_def_Brieskorn}

The Brieskorn manifolds are defined as follows: Let $n$ be a natural number and $a=(a_0,\ldots, a_n)$ an $(n+1)$-tuple of integers $\geq 2$. Then the singular hypersurface
\[
V(a) \defeq \{z\in \CC^{n+1} \mid z_0^{a_0} + \cdots + z_n^{a_n}=0\}
\]
is called the \emph{Brieskorn variety} of $a$, and
\[
 \Sigma(a)=V(a)\cap S^{2n+1}
\]
the \emph{Brieskorn manifold}. Moreover, the one-form
\begin{equation} \label{eq_contact_form}
 \alpha_a = \frac{i}{8} \sum_{j=0}^n a_j (z_j d\bar z_j - \bar z_j dz_j) 
\end{equation}
is a contact form on $\Sigma(a)$ (see \cite{Lutz_Meckert}), with Reeb vector field
\[
 R_{\alpha_a} = \left(\frac{4i}{a_0} z_0, \ldots, \frac{4i}{a_n} z_n\right). \text{\footnotemark}
\]
\footnotetext{By this notation, we actually mean the vector field $\sum_j \left(\frac{4i}{a_j}z_j\dd_{z_j} - \frac{4i}{a_j}\bar z_j\dd_{\bar z_j}\right)$. In particular, $R_\alpha$ lives in the real tangent space.}
Its flow is given by
\[
 \phi_t(z) = \left(e^{4it/a_0}z_0, \ldots, e^{4it/a_n}z_n\right).
\]

It is easy to find an exact symplectic filling of Brieskorn manifolds. Indeed, we can take the deformation 
\[
 V_\epsilon(a)\defeq \{z\in \CC^{n+1} \mid z_0^{a_0} + \cdots + z_n^{a_n}=\epsilon\}
\]
of $V(a)$ (with $\epsilon$ sufficiently small) and intersect it with the unit ball $B^{2(n+1)}$. The resulting manifold is smooth. Outside the origin, we can undo the deformation again, so that the boundary is just $\Sigma(a)$. 

A bit more precisely, we use a smooth, monotone decreasing cutoff function $\phi\in C^\infty(\RR)$ that fulfills $\phi(x)=1$ for $x\leq 1/4$ and $\phi(x)=0$ for $x\geq 3/4$. Then we define
\begin{equation} \label{eq_W}
 W = W_a = \left\{z\in \CC^{n+1} \mid z_0^{a_0} + \cdots + z_n^{a_n}=\epsilon \cdot \phi(\|z\|)\right\} \cap B^{2n+2}.
\end{equation}
As shown in \cite{Fauck}, this is an exact symplectic manifold $(W, \omega=d\theta)$, with boundary $\dd W=\Sigma(a)$ and $\theta|_{\dd W}=\alpha_a$. Alternatively, one could directly take $V_\epsilon(a)\cap B^{2n+2}$ as $W$ and use Gray's stability theorem to see that its boundary is contactomorphic to $\Sigma(a)$.

In Section~\ref{sec_sympl_homol}, we will examine a very special class of Brieskorn manifolds, namely those with $n=2m+1$ odd and $a=(2\ell,2,\ldots, 2)$. We abbreviate them by 
\[
 \Sigma_\ell\defeq \Sigma(2\ell,2,\ldots,2).
\]
We see immediately that in this case, the formulas for the contact form, the Reeb vector field and its flow simplify to
\begin{align*}
 \alpha &= \frac{i\ell}{4}(z_0 d\bar z_0 - \bar z_0 dz_0) + \frac{i}{4} \sum_{j=1}^n(z_j d\bar z_j - \bar z_j dz_j) \\
 R_\alpha &= 2i \left(\ell^{-1}z_0, z_1, \ldots, z_n\right) \\
 \phi_t(z) &= \left(e^{2it/\ell }z_0, e^{2it}z_1, \ldots, e^{2it}z_n\right).
\end{align*}

\subsection{Topology of Brieskorn manifolds} \label{sec_topo}

The singular homology of Brieskorn manifolds is very well understood. It is a classical fact (see e.g.\ \cite[Theorem~5.2]{Milnor}) that $\Sigma$ is highly-connected, meaning that 
\[
 \pi_1(\Sigma) = \cdots = \pi_{n-2}(\Sigma) = 0
\]
Consequently, their homology is concentrated in degrees $0,n-1,n,2n+1$. Of course, $H_0(\Sigma)\cong H_{2n+1}(\Sigma)\cong \ZZ$.
The homology in the middle dimension can be computed by a combinatorial algorithm from Randell \cite{Randell}. In the case of $\Sigma_\ell $, this algorithm gives $H_{n-1}(\Sigma_\ell )\cong \ZZ$, with no torsion elements. 

If one only wants to know whether $\Sigma(a)$ is a homology sphere, there is a simpler criterion, which was already known to Brieskorn (see \cite[Satz 1]{Brieskorn}).

\begin{thm}[Brieskorn] \label{thm_topol_sphere}
 The Brieskorn manifold $\Sigma(a_0, \ldots, a_n)$, $n\geq 3$, is a topological sphere if and only if one of the following two conditions holds:
 \begin{enumerate}[(i)]
  \item \label{item_two_isolated_pts} There are two exponents $a_i, a_j$ which are relatively prime to all the other exponents.
  \item \label{item_condition_B} There is one exponent $a_i$ which is relatively prime to all the other exponents. Additionally, there is a set of exponents $a_{j_1}, \ldots, a_{j_r}$, with $r\geq 3$ odd, such that each $a_{j_k}$ is relatively prime to any exponent not in the set, while $\gcd(a_{j_k}, a_{j_\ell })=2$ for all $k\neq \ell $.
 \end{enumerate}
\end{thm}

The case $n=2$ is excluded in this theorem because then, the manifold is no longer simply-connected.
For all examples in Section~\ref{sec_exotic_S7}, condition \eqref{item_two_isolated_pts} will be satisfied. 

Another classical result concerns the topology of the filling $W_a$ from \eqref{eq_W}. By \cite[Theorems 5.1 and 6.5]{Milnor}, it is parallelizable and has the homotopy type of a wedge of $\mu=\mu(a) = \prod_{i=0}^n (a_i-1)$ copies of $S^n$.

Once one knows the topology, one can ask for the diffeomorphism type of a Brieskorn manifold. This can sometimes be deduced from Wall's classification of highly-connected manifolds \cite{Wall}. The following result about the diffeomorphism type of $\Sigma_\ell $ can be found in \cite[Proposition~6.1]{Durfee-Kauffman}.

\begin{prop}
 Fix $n\geq 3$ odd. Denote by $K$ the Kervaire sphere of dimension $2n-1$ (which can be defined as the Brieskorn manifold $\Sigma(3,2,\ldots,2)$) and by $S^*S^n$ the unit cotangent bundle of $S^n$. The diffeomorphism type of $\Sigma_\ell $ is given as follows:
 \begin{equation*}
  \Sigma_\ell  \cong 
  \begin{cases}
   S^{n-1}\times S^n & \text{ if } \quad \ell\equiv 0 \mod 4 \\
   S^*S^n & \text{ if } \quad \ell\equiv 1 \mod 4 \\
   (S^{n-1}\times S^n) \# K & \text{ if } \quad \ell\equiv 2 \mod 4 \\
   S^*S^n \# K & \text{ if } \quad \ell\equiv 3 \mod 4 \\
  \end{cases}
 \end{equation*}
\end{prop}

In dimension $5$, the Kervaire sphere is diffeomorphic to the standard sphere \cite[Lemma~7.2]{KerMil}. Moreover, the cotangent bundle of $S^3$ is trivial, so $S^*S^3 \cong S^2\times S^3$, and we get:

\begin{cor}
 $\Sigma(2\ell,2,2,2)$ is diffeomorphic to $S^2\times S^3$.
\end{cor}

If $\Sigma(a)$ is a topological sphere, it represents an element of the group $bP_{2n}$ of boundary-parallelizable homotopy spheres. For even $n$, this element can be identified from the signature of the filling $W_a$, as will be explained in Section~\ref{sec_diffeo_type}.

\section{Symplectic homology of $\Sigma(2\ell,2,\ldots,2)$} \label{sec_sympl_homol}

\subsection{Full Morse--Bott formalism} \label{sec_MB}

As usual in symplectic homology, we take the completion $\widehat W$ of the Liouville domain $(W,\theta)$. To construct it, denote by $Z$ the Liouville vector field, defined by $\iota_Z\omega = \theta$, and by $\rho_t$ its flow. A neighborhood $U$ of $\Sigma =\dd W\subset W$ can be parametrized by $\psi\colon [-\delta, 0]\times \Sigma \to U, \quad (r,x)\mapsto \rho_r(x)$. Then the symplectic completion is defined as the manifold
\[
 \widehat W \defeq W \cup_\psi (\RR_{\geq 0} \times \Sigma),
\]
equipped with the symplectic form
\[
 \hat \omega = d\hat \theta, \qquad \hat \theta \defeq
 \begin{cases}
  \theta \qquad & \text{on } W \\
  e^r \theta \qquad & \text{on } \RR_{\geq 0} \times \Sigma,
 \end{cases}
\]
where $r$ is the coordinate on $\RR_{\geq 0}$.
Note that $\hat \theta$ is a smooth one-form because $\psi^*(\theta)=e^r \alpha$. 

Next, we take a Hamiltonian $H$ on $\widehat W$ which is $C^2$-small on $W$ and has the form $H|_{\RR_{\geq 0} \times \dd W} = h(e^r)$ for some strictly increasing function $h$ satisfying $\lim_{x\to \infty}h'(x)=\infty$ and $h''>0$ (see \cite{Seidel} for some background about this approach). Its Hamiltonian vector field, defined by $dH=-\iota_{X_H} \hat \omega$, is
\[
 X_H(r,x) = h'(e^r) R_\alpha(r,x)
\]
for any point $(r,x)\in \RR_{\geq 0} \times \Sigma$. Hence, the $1$-periodic orbits of $X_H$ are either
\begin{itemize}
 \item critical points of $H$ in $W$, or
 \item 1-periodic orbits on the level sets $\{r\} \times \Sigma$, which are in one-to-one correspondence with closed Reeb orbits of period $h'(e^r)$ on $\Sigma$.
\end{itemize}
The orbits of the first kind give the negative part of symplectic homology, $SH_*^-(W)$, which is well-known to be isomorphic to the relative singular homology $H_{n+*}(W,\Sigma)$.
Thus we will now focus on the positive part of symplectic homology $SH_*^+(W)$, generated by the closed Reeb orbits in $\Sigma$.

In the usual approach to symplectic homology, one works with a (time-dependent) Hamiltonian for which the $1$-periodic orbits are non-degenerate.
This is clearly not the case for our choice of $H$. However, there is a Morse--Bott approach to symplectic homology, worked out in \cite{BO_mb}, which deals with Hamiltonians with degenerate orbits. While \cite{BO_mb} considered only the case of Hamiltonians for which the $1$-periodic orbits are transversally non-degenerate (i.e.\ appear only in $S^1$-families), it turns out that analogous statements are true for more general Hamiltonians.

For this approach, the relevant conditions are:
\begin{itemize}
\item The space 
\[
 \mcN_T\defeq \{z\in\Sigma \mid \phi_T(z)=z \}
\]
consisting of $T$-periodic orbits is a closed submanifold for any $T\in\RR_{\geq 0}$, such that the rank $d\alpha|_{\mcN_T}$ is locally constant and $T_p \mcN_T=\ker (d_p\phi_T-\id)$.
\item The set $\{T\geq 0 \mid \mcN_T\neq \emptyset\}$ is discrete.\begin{footnote}
{
Actually, ongoing work by Fauck \cite{Fauck_prep} indicates that the second condition follows from the first, so it might be redundant.
However, both conditions are obvious for Brieskorn manifolds, so we choose not to worry about this.
}
\end{footnote}
\end{itemize}
These conditions guarantee that we have a Morse--Bott setting for symplectic homology, with the critical submanifolds $\mcN_T$.
Moreover, to have a well-defined grading, assume that 
\begin{itemize}
 \item $c_1(W)=0$ and
 \item all closed Reeb orbits of $\Sigma$ are contractible in $\Sigma$.
\end{itemize}
The second assumption makes sure that the grading of generators of $SH^+(W)$ is independent of the filling $W$ (provided that $c_1(W)=0$). For Brieskorn manifolds of dimension at least five, both conditions are clearly satisfied.

Next, we choose a Morse function $f_T$ on each (non-empty) $\mcN_T$. Then, the generators of the complex $SC^+(W)$ are given by pairs $(T,\eta)$, where $\eta$ is a critical point of $f_T$.
Its index is given by (see \cite[Lemma~2.4]{Bourgeois_thesis})
\begin{equation} \label{eq_index}
 \mu(T,\eta) = \mu_{\RS}(\mcN_T) + \ind(\eta) - \frac{1}{2}(\dim \mcN_T-1),
\end{equation}
where $\ind(\eta)$ is the Morse index of $\eta$ as a critical point of $f_T:\mcN_T\rightarrow \RR$, and $\mu_{\RS}(\mcN_T)$ is the Robbin--Salamon index of the path of symplectic matrices induced by an orbit in $\mcN_T$, as described in \cite{Robbin-Salamon}. (Some authors call $\mu_{RS}$ the Maslov index, but this terminology is somewhat ambiguous.)

\medskip

Now, we apply this setup to the Brieskorn manifolds $\Sigma_\ell$ for $n\geq 3$ odd. The Reeb flow $\phi_t$ is periodic everywhere. Moreover, if we start at a point $z\in \Sigma_\ell $, its minimal period is $\pi$ if $z_0=0$ and $\ell\pi$ otherwise. Hence we get the critical submanifolds
\begin{equation*}
 \mcN_T = 
 \begin{cases}
  \Sigma_\ell  & \text{ if } T=N\pi, N\in\NN, \ell\mid N \\
  \{z\in\Sigma_\ell  \mid z_0=0\} & \text{ if } T=N\pi, N\in \NN, \ell\nmid N \\
  \emptyset & \text{ else.}
 \end{cases}
\end{equation*}
(In our convention, $\NN = \{1,2,\ldots\}$.)
As for the Morse functions, first note that
\[
 \left\{z\in\Sigma_\ell  \mid z_0=0\right\} = \left\{z\in\CC^{n+1}\mid z_0=0, \sum_{j=1}^n z_k^2=0, |z|^2=1 \right\}
\]
is diffeomorphic to the unit cotangent bundle of $S^{n-1}$ \cite[Section~3.2]{Fauck}. As shown in the appendix of \cite{Fauck}, there exists a perfect Morse function (for $\ZZ_2$-coefficients) on $S^*S^{n-1}$, i.e.\ a Morse function with only four critical points with indices $0,n-2,n-1,2n-3$. If $n=3$, we can use the same function on $\Sigma_\ell $. For $n>3$, the existence of a perfect Morse function on $\Sigma_\ell $ is not obvious. However, ongoing work of Fauck \cite{Fauck_prep} shows that, as least with field coefficients, one can formally work with the chain complex as if one had a perfect Morse function. Roughly, the argument is that the generators fit into a spectral sequence whose first page consists of the homology of the critical submanifolds (with appropriate degree shifts). Hence, we pretend to have a function on $\Sigma_\ell $ whose only critical points have indices $0,n-1,n,2n-3$.

The Robbin--Salamon-indices of these submanifolds have been computed by Fauck \cite{Fauck}, and in a slightly different notation earlier by van Koert \cite{vanKoert}.
For a general Brieskorn manifold, all periods $T$ are multiples of $\frac{\pi}{2}$, and the Robbin--Salamon index of the critical submanifold $\mcN_{L \pi/2}$ is
\begin{equation} \label{eq_Maslov_index}
 \mu_{\RS}(\mcN_{L\pi/2}) = \sum_{j=0}^n \left(\floor*{\frac{L}{a_j}}+\ceil*{\frac{L}{a_j}}\right) - 2L.
\end{equation}
For $a=(2\ell,2\ldots, 2)$, all periods are in fact multiples of $\pi$. So for $T=N \pi$, i.e.\ $L=2N$, this formula specializes to
\begin{align*}
 \mu_{\RS}(\mcN_{N\pi}) &= \floor*{\frac{N}{\ell }} + \ceil*{\frac{N}{\ell }} + 2N(n-2) \\
 &= \begin{cases}
     2\frac{N}{\ell } + 2N(n-2) & \mbox{if  } \ell\mid N \\
     2\floor*{\frac{N}{\ell }} + 2N(n-2)+1 & \mbox{if  } \ell\nmid N. \\
    \end{cases}
\end{align*}

When plugging this, together with the Morse indices mentioned above, into \eqref{eq_index}, the chain groups of the positive part of symplectic homology (with coefficients in $\ZZ_2$) take the form
\begin{equation} \label{eq_SC_plus_mb}
 SC^+_k(W) \quad \cong \bigoplus_{\substack{N\in\NN,\ \ell\mid N \\ d\in\{-n+1,0,1,n\}}} (\ZZ_2)_{2N/\ell +2N(n-2)+d} \quad \oplus \bigoplus_{\substack{N\in\NN,\ \ell\nmid N \\ d\in\{-n+3,1,2,n\}}} (\ZZ_2)_{2\floor{N/\ell} +2N(n-2)+d}
\end{equation}
where $(\ZZ_2)_k$ denotes the $\ZZ_2$-vector space on one generator of degree $k$.

\medskip

As the differential decreases the degree by one, many differentials can already be excluded by degree reasons.
Note also that there cannot be a differential between $2N/\ell +2N(n-2)+1$ and $2N/\ell +2N(n-2)$ or between $2\floor{N/\ell }+2N(n-2)+2$ and $2\floor{N/\ell }+2N(n-2)+1$, as the corresponding generators lie on the same critical manifold and the Morse--Bott differential is zero (as we had a perfect Morse function). Thus, we only have to check whether there can be non-zero differentials between indices with different values for $N$. 

Now, we assume $n\geq 5$. (The case $n=3$ will be examined in the next section.) Then, an easy computation shows that each degree occurs at most once and we have precisely one pair of generators in consecutive degrees for any pair $(N,N+1)$. However, there is an easy way to see that the differential vanishes. This is because the index of the generator with period $(N+1)$ is lower by one than the index of the generator with period $N$. By a standard argument of Floer homology, the differential cannot increase the period. 
To see this, remember that the generators of $SH^+$ are critical points of the action functional
\[
 \mcA_H:\mcC^\infty(S^1, \widehat W)\rightarrow \RR, \qquad \mcA_H(\gamma) = \int_{S^1}\gamma^*\hat\theta - \int_{S^1} H(t,\gamma(t))\, dt.
\]
Its value at a critical orbit $\gamma$ is $\mcA_H(\gamma)=e^rh'(e^r)-h(e^r)$. Note that $\dd_r(e^rh'(e^r)-h(e^r)) = e^{2r} h''(e^r) >0$, so Hamiltonian orbits with larger action correspond to larger values of $r$, and thus to Reeb orbits with larger period. The differential $\dd$ counts negative gradient trajectories of the action functional. Hence, $\dd$ decreases the action, as well as the periods of the corresponding Reeb orbits.

This concludes that the positive symplectic homology is generated by the chains above:

\begin{thm} \label{thm_ngeq5}
 For $n\geq5$, the positive symplectic homology of the filling of $\Sigma_\ell $ is given by
 \begin{equation*}
 SH^+_k(W) \cong
 \left\{
 \begin{aligned}
  \ZZ_2 & \qquad \mbox{if } k = 2N/\ell +2N(n-2)-n+1 & \\
  & \qquad\mbox{or } k = 2N/\ell +2N(n-2) & \\
  & \qquad\mbox{or } k = 2N/\ell +2N(n-2)+1 & \\
  & \qquad\mbox{or } k = 2N/\ell +2N(n-2)+n & & \mbox{for some }N\in\NN, \ell\mid N \\
  \ZZ_2 & \qquad\mbox{if } k = 2\floor{N/\ell }+2N(n-2)-n+3 & \\
  & \qquad\mbox{or } k = 2\floor{N/\ell }+2N(n-2)+1 & \\
  & \qquad\mbox{or } k = 2\floor{N/\ell }+2N(n-2)+2 & \\
  & \qquad\mbox{or } k = 2\floor{N/\ell }+2N(n-2)+n & & \mbox{for some } N\in \NN, \ell\nmid N \\
  0 & \qquad\mbox{else.}
  \end{aligned}
 \right.
\end{equation*}
\end{thm}

The next lemma implies that we can use this theorem to distinguish the contact structures on $\Sigma_\ell $. Alternatively, we could argue that $\Sigma_\ell $ and $\Sigma_{\ell '}$ are non-contactomorphic because there is a degree $k\geq n+2$ in which the filling $\Sigma_\ell $ has non-vanishing symplectic homology, while, for a suitable non-degenerate contact form, there is no Reeb orbit with Conley--Zehnder index $k$ on $\Sigma_{\ell '}$.

\begin{lem} \label{lem_indep_filling}
 For $\Sigma_\ell$, the positive part of symplectic homology $SH^+$ is independent of the Liouville filling $W$, as long as $c_1(W)|_{\pi_2(W)}=0$.
\end{lem}

\begin{proof}
 The proof is analogous to Theorem~1.14 in \cite{Ciel_Fra_Oan}. Let us recall it.
 
 First of all, the generators are obviously independent of the filling, and by the assumption on $c_1$ and simple-connectedness, their grading is also well-defined and independent.
 
 Note that by \eqref{eq_SC_plus_mb} and \cite[Lemma~2.3]{Bourgeois_thesis}, there exists a perturbed contact form on $\Sigma_\ell $ such that all Reeb orbits are non-degenerate and have Conley--Zehnder indices 
 \begin{equation} \label{eq_CZ_big_enough}
 \mu_{\CZ}(\gamma)>3-n.
 \end{equation}
 We will show that this condition makes sure that the differential is also independent of the filling. Consider two generators $\overline \gamma_p, \underline \gamma_q$, i.e.\ $p$ and $q$ are critical points on the closed Reeb orbits $\overline \gamma, \underline \gamma$, with $\im(\overline \gamma) \neq \im(\underline \gamma)$ and $\mu(\overline\gamma_p) = \mu(\underline \gamma_q)+1$. 
 The main point in the proof is to use a stretching-the-neck operation as in \cite[Section~5.2]{BO_seq}. Under this operation, rigid Floer cylinders between $\overline \gamma$ and $\underline \gamma$ in $\widehat W$ turn bijectively into certain Floer buildings which live partly in the symplectization $\RR \times \Sigma$ and partly in the filling $W$. Each of these parts may have several levels.
 
 For the purpose of this lemma, we are only concerned with the level at the top, which is a punctured Floer cylinder $\tilde u$ in the symplectization $\RR\times \Sigma$ between the orbits $\overline \gamma$ and $\underline \gamma$. At each puncture, $\tilde u$ is asymptotic to further orbits $\gamma_1, \ldots, \gamma_k$. By using an almost complex structure $J$ which is time-dependent near the orbits $\overline \gamma, \underline \gamma$, we can assume that $\tilde u$ is cut out transversally.
 
 However, $\tilde u$ lives in a moduli space of virtual dimension
 \[
  \mu(\overline \gamma_p) - \mu(\underline \gamma_q) - \sum_{j=1}^k \left(\mu_\CZ(\gamma_j)+n-3\right) - 1 = \sum_{j=1}^k \left(\mu_\CZ(\gamma_j)+n-3\right).
 \]
 Hence, by \eqref{eq_CZ_big_enough} and transversality, this moduli space can be non-empty only if $k=0$, i.e.\ if there are no punctures. Thus, the count of rigid Floer cylinders between $\overline \gamma$ and $\underline \gamma$ is independent of the filling.
\end{proof}
 Note that this proof does not rely on the well-definedness of (cylindrical or linearized) contact homology.

\begin{cor}
 For $n\geq 5$, the manifolds $\Sigma_\ell =\Sigma(2\ell,2,\ldots,2)$ are pairwise non-contactomorphic.
\end{cor}

\begin{rmk}
 In this section, one could have chosen coefficients in some other field instead of $\ZZ_2$. The only change would be in the homology of the critical submanifolds, and accordingly in $SH_*^+(W)$. Presumably, there is also a similar theorem for integer coefficients. However, this raises some difficulties because the critical submanifolds may not admit perfect Morse functions (e.g.\ $S^*S^{n-1}$ does not), and the spectral sequence argument from \cite{Fauck_prep} does not work over the integers.
\end{rmk}

\begin{rmk} \label{rmk_full_SH}
From here, one can easily compute the full symplectic homology of the filling $W$ of $\Sigma_\ell $.
Indeed, the singular relative homology of the pair $(W,\Sigma_\ell )$ is
\[
 H_k(W,\Sigma_\ell ;\ZZ) \cong 
 \begin{cases}
  \ZZ^{2\ell-1} &\text{ if } k=n \\
  \ZZ & \text{ if } k=2n \\
  0 & \text{ else},
 \end{cases}
\]
as can be seen from the statements at the beginning of Section~\ref{sec_topo}. Then, one uses $SH^-_k(W) \cong H_{k+n}(W,\Sigma_\ell )$ and the long exact sequence 
\[
 \cdots \longrightarrow SH^-_*(W) \longrightarrow SH_*(W) \longrightarrow SH^+_*(W) \longrightarrow SH^-_{*-1}(W) \longrightarrow \cdots 
\]
coming from the tautological exact sequence $0\to SC^-_* \to SC_* \to SC^+_* \to 0$. One can read off $SH_*(W)$ directly from this sequence, without having to know any of the maps.

Furthermore, one can easily compute the Rabinowitz Floer homology $RFH_*(W)$, either directly with the Morse--Bott methods used here (and in \cite{Fauck}) or from $SH_*(W)$, using the long exact sequence from \cite{Ciel_Fra_Oan}.
\end{rmk}

In the remaining part of Section~\ref{sec_sympl_homol}, we specialize to the case $n=3$. In this case, the methods considered up to this point are clearly insufficient to compute the symplectic homology, as there are many generators in consecutive degrees. The goal is to still get enough information to distinguish the contact structures for different values of $\ell$.

\subsection{Ustilovsky's perturbation} \label{sec_perturbation}

It turns out to be convenient to leave the full Morse--Bott formalism and work instead in a perturbed setup. We will use the same perturbation as Ustilovsky in \cite{Ustil}. In fact, we are still in a Morse--Bott situation after the perturbation, but with all critical manifolds being $S^1$. This is exactly the setup used in \cite{BO_mb}.

To start, Ustilovsky changes the coordinates by the following unitary transformation:
\[
 w_0=z_0, \qquad w_1=z_1, \qquad
 \begin{pmatrix}
  w_2 \\ w_3
 \end{pmatrix}
 = \frac{1}{\sqrt{2}}
 \begin{pmatrix}
  1 & i \\ 1 & -i
 \end{pmatrix}
 \begin{pmatrix}
  z_2 \\ z_3
 \end{pmatrix}.
\]
In these coordinates, 
\[
 \Sigma_\ell  = \left\{w\in\CC^4 \mid w_0^{2\ell}+w_1^2+2 w_2 w_3=0, \|w\|^2=1\right\}.
\]
Next, Ustilovsky introduces a new contact form $\alpha'\defeq K^{-1}\alpha$, where
\[
 K(w)\defeq \|w\|^2 + \epsilon \left(|w_2|^2-|w_3|^2\right)
\]
and $\epsilon>0$ is a sufficiently small, irrational number.\footnote{For the sake of simplicity, we have, as Ustilovsky, perturbed the contact form. In fact, it is possible to get the same outcome by perturbing the Hamiltonian, as is more common in symplectic homology, although it cannot be written down as nicely. To perturb the Hamiltonian orbits on the level set $\{r_0\}\times\Sigma_\ell$, one has to add $H_{\mathrm{pert},r_0}\defeq \epsilon h'(e^{r_0})e^r(|w_2|^2-|w_3|^2)$ to the Hamiltonian. Doing this for all level sets containing critical submanifolds (with suitable cutoff functions), one gets that the Hamiltonian vector field equals $h'(e^r)\cdot R_{\alpha'}$ near the critical submanifolds, hence the perturbed orbits are the same.}
He then shows that the corresponding Reeb vector field is
\begin{equation*}
 R_{\alpha'}= \left(\frac{2i}{\ell }w_0, 2i w_1, 2i(1+\epsilon)w_2, 2i(1-\epsilon_1)w_3 \right).
\end{equation*}
Hence, the only simple (i.e.\ not multiply covered) periodic Reeb orbits are
\begin{align} 
 \gamma^{0,+}(t) &= \left(re^{2it/\ell }, ir^\ell e^{2it},0, 0\right), \qquad r>0, r^{2\ell}+r^2=1,\quad 0\leq t\leq \ell\pi, \label{ustil_orbits_0} \\
 \gamma^{0,-}(t) &= \left(re^{2it/\ell }, -ir^\ell e^{2it},0, 0\right), \qquad r>0, r^{2\ell}+r^2=1,\quad 0\leq t\leq \ell\pi, \label{ustil_orbits_1} \\
 \gamma^+(t) &= \left(0, 0, e^{2it(1+\epsilon)}, 0\right), \qquad 0\leq t\leq\frac{\pi}{1+\epsilon}, \label{ustil_orbits_2} \\
 \gamma^-(t) &= \left(0, 0, 0, e^{2it(1-\epsilon)}\right), \qquad 0\leq t\leq\frac{\pi}{1-\epsilon}, \label{ustil_orbits_3}
\end{align}
and their multiples, all of which are transversally non-degenerate. Furthermore, for $\epsilon$ sufficiently small, the Conley--Zehnder indices of these orbits and their multiple covers (denoted by $N\gamma$ for a simple orbit $\gamma$) are given by
\begin{align}
 \mu_\CZ(N\gamma^{0,\pm}) &= 2N+2N\ell \quad \stackrel{N'\defeq N\ell }{=} \quad 2\frac{N'}{\ell }+2N' \label{ustil_ind_1}\\
 \mu_\CZ(N\gamma^+) &= 2\ceil*{\frac{N}{\ell }} +2N-2 \label{ustil_ind_2} \\
 \mu_\CZ(N\gamma^-) &= 2\floor*{\frac{N}{\ell }} +2N+2, \label{ustil_ind_3}
\end{align}
by a computation analogous to \cite[Lemma~4.2]{Ustil}.

Applying the Morse--Bott formalism of Section~\ref{sec_MB} to this situation, we get two generators for each of the orbits above. We denote these by $\gamma_m$ and $\gamma_M$, with $\mu(\gamma_M)=\mu(\gamma_m)+1$.\footnote{
Due to our use of \emph{negative} gradient flow lines, the indices of minimum and maximum are interchanged compared to \cite{BO_mb}.}
Hence we get the generators of $SC^+(W)$ as in Table~\ref{table_generators}.

\bigskip

\begin{table}[ht]
\begin{tabulary}{\textwidth}{L|C|C|C|C|C|C|C|C|C|C|C|C|C}
\hline
 Degree & $2$ & $3$ & $4$ & $5$ & $6$ & $7$ & $8$ & $\cdots$ & $2\ell~+~1$ & $2\ell~+~2$ & $2\ell~+~3$ & $2\ell~+~4$ & $\cdots$ \\
\hline
 Generators & $\gamma^+_m$ & $\gamma^+_M$ & $2\gamma^+_m$ & $2\gamma^+_M$ & $3\gamma^+_m$ & $3\gamma^+_M$ & $4\gamma^+_m$ & $\cdots$ & $\ell\gamma^+_M$ & $\gamma^{0,+}_m$ & $\gamma^{0,+}_M$  & $(\ell~+~1)\gamma^+_m$ & $\cdots$\\
 & & & $\gamma^-_m$ & $\gamma^-_M$ & $2\gamma^-_m$ & $2\gamma^-_M$ & $3\gamma^-_m$ & $\cdots$ & $(\ell~-~1)\gamma^-_M$ & $\gamma^{0,-}_m$ & $\gamma^{0,-}_M$  & $\ell\gamma^-_m$ & $\cdots$ \\
\hline
\end{tabulary}
\caption{The generators of $SC^+(W)$ for $n=3$, in the perturbed Morse--Bott setup.}
\label{table_generators}
\end{table}

Let us recall at this point how the differential in \cite{BO_mb} was defined.
Denote by $S_\gamma$ the $S^1$-family of orbits with geometric image $\im(\gamma)$. Given a Hamiltonian $H$ as in Section~\ref{sec_MB} and an $\hat\omega$-compatible, time-dependent almost complex structure $J$, the set $\widehat \mcM(S_{\overline \gamma}, S_{\underline \gamma}; H, J)$ is defined as the space of cylinders $u:\RR\times S^1 \to \widehat W$ satisfying the Floer equation
\begin{equation} \label{eq_Floer}
 \dd_s u + J_{(\theta,u)}(\dd_\theta u-X_H(u)) = 0,
\end{equation}
which converge asymptotically to some orbits in $S_{\overline \gamma}$ and $S_{\underline \gamma}$. The last part means that there exist orbits $\overline \gamma \in S_{\overline \gamma}$, $\underline \gamma \in S_{\underline \gamma}$ such that
\[
 \lim_{s\to -\infty} u(s,\theta)=\underline\gamma(\theta), \qquad \lim_{s\to \infty} u(s,\theta) =\overline\gamma(\theta), \qquad \lim_{s\to\pm\infty}\dd_s u(s,\theta)=0,
\]
uniformly in $\theta$.

If $S_{\overline\gamma} \neq S_{\underline \gamma}$, there is a free $\RR$-action on $\widehat \mcM(S_{\overline \gamma}, S_{\underline \gamma}; H, J)$, defined by $s_0\cdot u(\cdot, \cdot)=u(s_0+\cdot, \cdot)$. Define the moduli space as $\mcM(S_{\overline \gamma}, S_{\underline \gamma}; H, J) \defeq \widehat \mcM(S_{\overline \gamma}, S_{\underline \gamma}; H, J)/\RR$.\footnote{The homology class represented by $u$ is not specified because $H_2(\widehat W)=0$.}

These moduli spaces come along with evaluation maps $\overline\ev, \underline\ev$, defined by $u\mapsto \overline \gamma(0)$ and $u\mapsto \underline \gamma(0)$, respectively.
For $J$ in a comeagre set, the moduli spaces $\mcM(S_{\overline \gamma}, S_{\underline \gamma}; H, J)$ are transversally cut out (see \cite[Theorem~3.5]{BO_mb}). 
In this case, we can choose perfect Morse functions $f_\gamma$ on the spaces $S_{\gamma}$ such that their stable and unstable manifolds (denoted by $W^s$ and $W^u$, respectively) are transverse to $\overline \ev$ and $\underline \ev$ (see \cite[Lemma~3.6]{BO_mb}). The minima and maxima of these Morse functions give the generators in Table~\ref{table_generators}.

For two generators $\overline \gamma_p, \underline\gamma_q$ with $S_{\overline\gamma} \neq S_{\underline \gamma}$, the fibered product
\begin{equation}\label{eq_fib_prod}
 \mcM(\overline \gamma_p, \underline \gamma_q; H,J) \defeq W^s(p) \times_{\overline \ev} \mcM(S_{\overline \gamma}, S_{\underline \gamma}; H,J)_{\underline\ev} \times W^u(q).
\end{equation}
is a smooth, compact manifold of dimension $\mu(\overline \gamma_p)-\mu(\underline\gamma_q)-1$. In particular, if $\mu(\overline \gamma_p) - \mu(\underline\gamma_q) = 1$, it is a finite set. The coefficient $\langle \dd(\overline \gamma_p), \underline \gamma_q\rangle$ of the differential is then defined as the count (modulo $2$) of its elements.

If $S_{\overline\gamma} = S_{\underline \gamma}$, the coefficient $\langle \dd(\overline \gamma_p), \underline \gamma_q\rangle$ of the differential is simply the corresponding coefficient of the Morse differential. In our case, since all critical manifolds are circles and we use $\ZZ_2$-coefficients, all these differentials vanish. (For integer or rational coefficients, this would be true only for \emph{good} Reeb orbits, see Section~\ref{sec_mean_Euler_general} for the definition. For our choice of contact form on $\Sigma_\ell $, all Reeb orbits are actually good, as can be checked from equations \eqref{ustil_ind_1} to \eqref{ustil_ind_3}.)

\bigskip

The next goal is to collect as much information as possible about the differential between the generators in Table~\ref{table_generators}. First, it follows from \cite[Proposition~2]{BO_seq}, that there is no differential between $\gamma_M$ and $\gamma_m$ for any $\gamma$, so that we only need the definition above for the differential. Another important observation is given by the next lemma.

\begin{lem} \label{lem_rank_1}
 In the degrees $2N(\ell +1)+j$, where $N\in\NN$ and $j\in\{-1,0,1,2\}$, the rank of $SH^+(W)$ is at most one. In particular, there are some non-trivial differentials in Table~\ref{table_generators} in these degrees.
\end{lem}

\begin{proof}
 The chain complex from the full Morse--Bott setup of Section~\ref{sec_MB} will give the same symplectic homology groups. Hence, the ranks of the chain groups give upper bounds. Checking the ranks in \eqref{eq_SC_plus_mb}, one sees that this upper bound is one in the degrees $2N(\ell +1)+j$.
\end{proof}

We claim that away from these degrees, all differentials vanish. The proof of this uses a $\ZZ_2$-symmetry of $\Sigma_\ell $ (and $W$) and will occupy the rest of Section~\ref{sec_sympl_homol}. 

\begin{rmk}
 Although we will not need this, let us point out a few cases where the vanishing of the differential also follows from other reasons. For instance, by \cite[Proposition~3 and Remark~14]{BO_seq}, there cannot be a non-zero differential between $\gamma_M$ and $\tilde \gamma_m$ for any orbits $\gamma \neq \tilde \gamma$ with $\mu(\gamma)=\mu(\tilde \gamma)$, at least if we assume that transversality as in \cite[Remark~9]{BO_seq} holds for $\Sigma_\ell $.
 Checking the degrees in Table~\ref{table_generators}, this means that there is no differential from an odd degree to an even degree.
 
 One can also argue that there is no differential from $N\gamma^-_m$ to $N\gamma^+_M$. This is because in the full Morse--Bott picture from Section~\ref{sec_MB}, the generators corresponding to $N\gamma^-_m$ and $N\gamma^+_M$ belonged to the same critical manifold. They can be viewed as originating from a perturbation thereof. If there were a differential between them, it would have shown up in Section~\ref{sec_MB} as a Morse differential on this critical manifold, which it did not, as the Morse function on each critical submanifold was perfect. See e.g.\ \cite[Theorem~5.2.2]{BPS} for the correspondence between the Morse--Bott formalism and the perturbed version.
\end{rmk}

\subsection{Idea of symmetries} \label{sec_idea}

Define the involutive isomorphism
\[
 \psi: \CC^4 \to \CC^4, \qquad \psi(w_0,w_1,w_2,w_3) = (-w_0,-w_1,w_2,w_3).
\]
From the definition of $\Sigma_\ell $, one sees immediately that $\psi$ leaves $\Sigma_\ell $ invariant (because the exponents $a_0=2\ell,a_1=2$ are even). For the same reason, the filling $W$, as defined in \eqref{eq_W}, is left invariant, as well as its completion. Moreover, $\psi$ preserves the contact (resp.\ symplectic) form on $\Sigma_\ell $ (resp.\ $\widehat W$).

We can view $\psi$ as the generator of a $\ZZ_2$ symmetry. Denote by $\mcJ_\symm$ the set of all time-dependent, $\hat\omega$-almost complex structures that are symmetric under $\psi$, i.e.\ $\psi_* J = J$. The idea behind this definition is that we apply $\psi$ to the Floer cylinders that appear in the differential. The hope is that these cylinders always come in pairs $u, \psi\circ u$, so that the differential with coefficients in $\ZZ_2$ vanishes. 

As a first step, the next lemma lets us assume that the fixed point set $\widehat W_\fixed$ of $\psi$ does not contain any Floer cylinders. 

\begin{lem} \label{lem_fixed}
Let $\overline \gamma_p$ and $\underline \gamma_q$ be two generators with $S_{\overline\gamma} \neq S_{\underline \gamma}$.
There are no elements of \eqref{eq_fib_prod} whose Floer cylinders are contained in the fixed point set $\widehat W_\fixed$ of $\psi$.
\end{lem}

\begin{proof}
 If one of the underlying orbits $\overline \gamma$ or $\underline \gamma$ lies outside of $\widehat W_\fixed$, there is nothing to show. So assume that both are multiples of $\gamma^+$ or $\gamma^-$. For this case, we are going to show that any two of these orbits live in distinct homotopy classes in $\widehat W_\fixed$.
 
 The fixed point set 
 \[
  \widehat W_\fixed = \{(w_2,w_3)\in\CC^2 \mid 2 w_2 w_3 = \epsilon \phi(\|w\|)\}
 \]
 is diffeomorphic to $\CC^*\cong\RR\times S^1$.
 Explicitly, this isomorphism can be taken to be the composition
 \[
  \widehat W_\fixed \longrightarrow \left\{w_2 w_3 = \phi\big(\sqrt{\epsilon/2}\|w\|\big)\right\} \longrightarrow \{w_2 w_3 = 1\} \longrightarrow \CC^*,
 \]
 where the first map is scaling by $\sqrt{2/\epsilon}$, the second is
 \[
 (w_2, w_3) \mapsto 
 \begin{cases}
  \big(w_2, \frac{1}{w_2}\big) & \text{if} \quad |w_2|\geq |w_3| \\
  \big(\frac{1}{w_3}, w_3\big) & \text{if} \quad |w_2|\leq |w_3|,
 \end{cases}
\]
and the third is the inverse of $z\mapsto (z, 1/z)$.
Therefore, the orbit $N\gamma^\pm$ is mapped to a loop homotopic to
\[
 [0,1] \to \CC^*, \qquad t\mapsto e^{\pm 2\pi N i t}.
\]
For different values of $N\in\NN$ and $\pm$, these loops all represent different elements of $\pi_1(\CC^*)\cong \ZZ$, hence there can be no cylinder in between.
\end{proof}

\begin{rmk}
 This simple proof was suggested by the referee. In the first version of this article, an alternative argument was given, based on the fact that the orbits $N\gamma^\pm$ have two different kinds of Conley--Zehnder indices: one with $\widehat W$ as the ambient manifold and one with $\widehat W_\fixed$ as the ambient manifold. It turns out that for some orbits, the difference of the latter indices is smaller than the difference of the former indices. Thus, for generic $J\in \mcJ_\symm$, the moduli space of Floer cylinders in $\widehat W_\fixed$ has negative dimension. 
\end{rmk}

\begin{prop} \label{prop_reparametrization}
 Let $J\in \mcJ_\symm$ and $u:\RR\times S^1\to \widehat W$ be a $J$-Floer cylinder between orbits $\overline \gamma \neq \underline \gamma$ which are contained in $\widehat W_\fixed$. Then $\psi\circ u$ is again a Floer cylinder between the same orbits. Moreover, there is no constant $s_0$ such that $\psi\circ u(s,\theta)=u(s+s_0,\theta)$ for all $(s,\theta)\in \RR\times S^1$.
\end{prop}

The second claim ensures that $\psi\circ u$ and $u$ are counted separately in the moduli space for the differential. 

\begin{proof}
 By assumption, both $J$ and $X_H$ are equivariant with respect to $\psi$. Thus, if $u$ satisfies the Floer equation \eqref{eq_Floer}, we can apply $\psi_*$ to both sides and get
 \begin{align*}
  0 &= \psi_*\left(\dd_s u + J_{(\theta,u)}(\dd_\theta u-X_H(u))\right) \\
  &= \dd_s (\psi\circ u) + J_{(\theta,\psi\circ u)} \left(\dd_\theta(\psi\circ u)-X_H(\psi\circ u)\right),
 \end{align*}
 establishing that $\psi\circ u$ is a Floer cylinder. Since $\overline \gamma$ and $\underline \gamma$ lie in $\widehat W_\fixed$, the asymptotics of $\psi\circ u$ and $u$ are the same.
 
 If there were such a constant $s_0$, we could use $\psi\circ \psi=\id$ to get
 \[
  u(s+2s_0,\theta)=\psi\circ u(s+s_0,\theta) = \psi\circ \psi \circ u(s,\theta)=u(s,\theta).
 \]
 So the function $s\mapsto u(s,\theta)$ would be periodic with period $2s_0$, but it also has a limit as $s\to \infty$. Since it is not constant, this implies $s_0=0$ and hence $\psi\circ u=u$. But $u$ cannot lie in $\widehat W_\fixed$ by Lemma~\ref{lem_fixed}, which gives a contradiction.
\end{proof}

We want to apply this proposition to show that certain Floer cylinders contributing to the differential of symplectic homology come in pairs, so that the differential vanishes in $\ZZ_2$. Before, though, we must show that there are almost complex structures in $\mcJ_\symm$ such that the relevant moduli spaces are cut out transversally.

\subsection{Transversality} \label{sec_transversality}

\begin{prop} \label{prop_transversality}
 Given $H$ as in Section~\ref{sec_MB} (time-independent), there exists a comeagre set $\mcJ_{\symm, \reg} \subset \mcJ_\symm$ for which all moduli spaces $\mcM(\overline \gamma_p, \underline \gamma_q; H,J)$ with $\mu(\overline \gamma_p) - \mu(\underline \gamma_q)=1$ and $S_{\overline\gamma} \neq S_{\underline \gamma}$ are transversally cut out.
\end{prop}

\begin{proof}
Fix two generators $\overline \gamma_p$ and $\underline \gamma_q$.
For the most part, we have to prove the existence of a comeagre set $\mcJ_{\symm, \reg}$ such that the moduli space $\mcM(S_{\overline \gamma}, S_{\underline \gamma}; H, J)$ appearing in the fibered product \eqref{eq_fib_prod} is transversally cut out. Then, the statement follows from a generic choice of Morse functions as in \cite[Lemma~3.6]{BO_mb}.

To prove this, much of the proof of \cite[Proposition~3.5 (ii)]{BO_mb} can be followed very closely.
We will only point out the parts that are different.
The most important difference is that for all sets of almost complex structures (like $\mcJ^\ell $, $\mcJ^\ell (H)$, etc.), we additionally demand that $J\in J_\symm$. We then denote the corresponding sets by $\mcJ^\ell _\symm$, $\mcJ^\ell _\symm(H)$, etc. 
 
 So we take the universal moduli space
 \[
  \mcM(S_{\overline \gamma}, S_{\underline \gamma}, H, \mcJ^\ell _\symm(H)) = \left\{(u,J) \; \middle| \; J\in \mcJ^\ell _\symm(H), u\in \mcM(S_{\overline \gamma}, S_{\underline \gamma}, H,J) \right\}.
 \]
 We want to prove that this space is transversally cut out. Then we define $\mcJ_{\symm, \reg}$ as the set of regular values of the projection to the second factor.
 
 As usual, $\mcM(S_{\overline \gamma}, S_{\underline \gamma}, H, \mcJ^\ell _\symm(H))$ can be written as the preimage $\bar \dd_H^{-1}(0)$ under the section $\bar \dd_H$ of a Banach vector bundle $\mcE \rightarrow \mcB \times \mcJ^\ell _\symm(H)$. We do not write the details here, as this part is entirely analogous to \cite{BO_mb}.
 
 It remains to show that the vertical differential 
 \begin{align*}
  D \bar \dd_H(u,J): T_u \mcB \times T_J\mcJ^\ell _\symm(H) & \longrightarrow \mcE_{(u,J)} \\
  (\zeta,Y) & \longmapsto D_u \zeta +  Y_\theta(u)(\dd_\theta u-X_H(u))
 \end{align*}
 is surjective. Again as in \cite{BO_mb}, $D_u$ (the linearization of the Cauchy--Riemann operator) is Fredholm, so the range of $D \bar \dd_H(u,J)$ is closed. We have to show that it is also dense, and this is where some differences to \cite{BO_mb} appear. 
 
 Let $\eta$ be in the cokernel of $D \bar \dd_H(u,J)$, which means
 \begin{equation} \label{eq_coker}
  \int_{\RR\times S^1} \langle \eta, D_u \zeta\rangle e^{d|s|} ds \, d\theta = 0, \qquad \int_{\RR\times S^1} \langle \eta, Y_\theta(u)(\dd_\theta u - X_H(u)) \rangle e^{d|s|} ds \, d\theta = 0
 \end{equation}
 for all $\zeta, Y$, where $d>0$ is some exponential weight. The first equation is still the same as in \cite{BO_mb}. It implies that, assuming $\eta\not\equiv 0$, the set $\{(s,\theta)\mid \eta(s,\theta)\neq 0\}$ is open and dense. Also, by \cite[Lemma~4.5]{FHS}, the set of \emph{regular points}
 \[
  R(u)\defeq \left\{(s,\theta)\in \RR\times S^1 \mid \dd_s u(s,\theta)\neq 0, u(s,\theta)\neq \overline \gamma(\theta),\underline \gamma(\theta), u(s,\theta)\notin u(\RR\setminus\{s\},\theta)\right\}
 \]
 is open and dense for any $u\in \bar \dd^{-1}(0)$. 
 
 Furthermore, we claim that the set 
 \[
  S(u) \defeq \left\{(s,\theta)\in \RR\times S^1 \mid \dd_s u(s,\theta)\neq 0, u(s,\theta)\neq \overline \gamma(\theta),\underline \gamma(\theta), \psi(u(s,\theta)) \notin u(\RR, \theta) \right\} 
 \]
 is open and dense. Indeed, this can be proven in exactly the same way as \cite[Lemma~4.5]{FHS}, one just has to replace $u$ by $\psi\circ u$ at the right places and use that $\im(u) \not\subset \widehat W_\fixed$ by Lemma~\ref{lem_fixed}.
 The upshot is that we can find a point $(s_0,\theta_0)\in R(u) \cap S(u)$ with $\eta(s_0,\theta_0)\neq 0$. 
 
 Now, it is always possible to choose a matrix $Y_{\theta_0}(u(s_0,\theta_0))\in T_{J(u(s_0,\theta_0))} \mcJ^\ell (H)$ which maps the vector $J(u(s_0,\theta_0))(\dd_\theta u - X_H(u))$ to $\eta(u(s_0,\theta_0))$ (see e.g.\ \cite[Lemma~3.2.2]{Mc-Sa}). Letting $\rho:S^1\times \widehat W \to [0,1]$ be a time-dependent cutoff function supported near $(\theta_0, u(s_0,\theta_0))$, we define $Y\defeq \rho\cdot Y_{\theta_0}(u(s_0,\theta_0))$. Then
 \begin{equation} \label{eq_eta_contr}
  \int_{\RR\times S^1} \langle \eta, Y_\theta(u)(\dd_\theta u - X_H(u)) \rangle e^{d|s|} ds \, d\theta \neq 0.
 \end{equation}

 However, $Y$ is a priori not contained in the tangent space to $J^\ell _\symm(H)$. For this, we have to make it symmetric under $\psi$. Hence we define $Y^\symm \defeq Y+\psi_*(Y)$. We want to show that \eqref{eq_eta_contr} is still true with $Y$ replaced by $Y^\symm$.
 
 By construction, $Y^\symm$ is supported near the two point $(\theta_0, u(s_0,\theta_0))$ and $(\theta_0, \psi(u(s_0,\theta_0))$.
 If $\psi(u(s_0,\theta_0))\notin \im(u)$, we are done, because $\psi_*(Y)$ does not affect \eqref{eq_eta_contr} (provided the cutoff function $\rho$ was chosen well). Otherwise, since $(s_0,\theta_0) \in  R(u) \cap S(u)$, we know that $\psi(u(s_0,\theta_0))=u(s_1,\theta_1)$ for some $\theta_1\neq \theta_0$. But since $\rho$ is time-dependent and localized near $\theta_0$, $\psi_*(Y)$ still does not affect \eqref{eq_eta_contr}. Thus, \eqref{eq_eta_contr} is indeed true with $Y$ replaced by $Y^\symm\in T_J \mcJ^\ell _\symm(H)$, which contradicts \eqref{eq_coker}.
 
 This shows that $D\bar \dd_H(u,J)$ is surjective, hence the universal moduli space is cut out transversally. Define the set $\mcJ_{\symm, \reg}$ as the set of regular values under its projection to the second factor. By Sard's theorem, this set is comeagre and by construction, $\mcM(S_{\overline \gamma}, S_{\underline \gamma}; H, J)$ is cut out transversally for $J\in \mcJ_{\symm, \reg}$.
\end{proof}

\subsection{Conclusion} \label{sec_conclusion}

\begin{cor} \label{cor_d_van}
 Let $\overline \gamma_p, \underline \gamma_q$ be two generators with $\mu(\overline \gamma_p) - \mu(\underline \gamma_q)=1$, such that the underlying orbits lie in $\widehat W_\fixed$. Then, the differential satisfies $\left\langle \dd(\overline \gamma_p), \underline \gamma_q \right\rangle = 0$.
\end{cor}

\begin{proof}
By Proposition~\ref{prop_transversality}, we can assume that $J$ is symmetric under $\psi$. Then, Lemma~\ref{lem_fixed} and Proposition~\ref{prop_reparametrization} tell us that the elements in \eqref{eq_fib_prod} come in pairs, namely by replacing the Floer cylinder $u\in \mcM(S_{\overline \gamma}, S_{\underline \gamma}; H,J)$ with $\psi\circ u$. Hence the algebraic count is an even number, and thus vanishes for $\ZZ_2$-coefficients.
\end{proof}

Looking again at Table~\ref{table_generators}, this shows that all differentials involving only the orbits $N\gamma^+$ and $N\gamma^-$ vanish. This proves the claim made at the end of Section~\ref{sec_perturbation}. Hence, the rank of symplectic homology in degrees $4, \ldots, 2\ell$ is two (and again in degrees $2\ell+5, \ldots, 4\ell+2$, etc).

Up to here, we already know enough to distinguish the contact structures of $\Sigma_\ell $ for different $\ell$, but we can get another observation almost for free:

\begin{lem} \label{lem_remaining_degrees}
 For $N\in \NN$ and $j\in \{-1,0,1,2\}$, the groups $SH^+_{(2\ell+2)N+j}(W)$ are isomorphic to $\ZZ_2$.
\end{lem}

\begin{proof}
 We already know from Lemma~\ref{lem_rank_1} that these groups can have at most rank $1$. To see that they do not vanish, define the map
 \[
 \tilde \psi: \CC^4 \longrightarrow \CC^4, \qquad \tilde \psi(w_0,w_1,w_2,w_3) = (w_0,-w_1,w_2,w_3).
 \]
 As with $\psi$, this map descends to a $\ZZ_2$-symmetry on $\widehat W$. Furthermore, it exchanges the orbits $N\gamma^{0,+}$ and $N\gamma^{0,-}$, while leaving all other orbits fixed. In analogy to Proposition~\ref{prop_transversality}, we can find an almost complex structure $\tilde J$ with $\tilde J=\tilde \psi_*(\tilde J)$ such that the moduli spaces $\mcM(S_{\overline \gamma},S_{\underline \gamma}; H,\tilde J)$ are regular if at least one of the orbits $\overline\gamma, \underline \gamma$ lies outside of the fixed point set of $\tilde \psi$ (so that Floer cylinders in the fixed point set are excluded).
 
 To justify the switch to a different almost complex structure, consider the continuation homomorphism $\Phi$ from $(H,J)$ to $(H,\tilde J)$. As we change only the almost complex structure, not the Hamiltonian, $\Phi$ can be represented in each degree by an invertible matrix. Moreover, $\Phi$ intertwines the differential, i.e.\ $\dd \circ \Phi = \Phi \circ \dd$. Hence, the new differential is a conjugation of the old one, and as such has the same rank. For this lemma, only the rank is of interest, so we can indeed switch to another (regular) almost complex structure. 
 
 With $\tilde J$ (and suitable Morse functions on $S_{N\gamma^{0,\pm}}$), we get that 
 \[
  \left\langle \dd N\gamma^{0,+}_m, N\ell \gamma^+_M \right\rangle = \left\langle \dd N\gamma^{0,-}_m, N\ell \gamma^+_M \right\rangle,
 \]
 as $\tilde \psi$ interchanges all contributing cylinders. Similarly, 
\[
 \left\langle \dd(N\gamma^{0,+}_m), N(\ell -1) \gamma^-_M \right\rangle = \left\langle \dd(N\gamma^{0,-}_m), N(\ell -1) \gamma^-_M \right\rangle,
\]
 hence the map 
 \[
  \dd:SC^+_{N(2\ell+2)}\cong (\ZZ_2)^2 \longrightarrow SC^+_{N(2\ell+2)-1}\cong (\ZZ_2)^2
 \]
 is represented by a matrix of the form
 $\begin{pmatrix}
  a & a \\
  b & b
 \end{pmatrix}
 \in \ZZ_2^{2\times 2}.$
 This matrix has rank at most one, but it cannot have rank zero, as this would contradict Lemma~\ref{lem_rank_1}.
 
 Now, note that we cannot have any differential from $N\gamma^{0,+}_M$ to $N\gamma^{0,-}_m$ or from $N\gamma^{0,-}_M$ to $N\gamma^{0,+}_m$. The easiest way to see this is that the underlying orbits have exactly the same period, hence exactly the same action, while the differential strictly decreases the action.
 
 This proves the lemma for $j=-1,0$. For $j=1,2$, note that, for the same reasons as above, 
 \[
  \dd:SC^+_{N(2\ell+2)+2}\cong (\ZZ_2)^2 \longrightarrow SC^+_{N(2\ell+2)+1}\cong (\ZZ_2)^2
 \]
 is represented by a matrix of the form
 $\begin{pmatrix}
  a & b \\
  a & b
 \end{pmatrix}
 \in \ZZ_2^{2\times 2}.$
 As this matrix also has rank at most one and rank zero would again contradict Lemma~\ref{lem_rank_1}, the claim follows. 
\end{proof}

Summing up, we have proven:

\begin{thm} \label{thm_n3}
 Let $W$ be a Liouville filling of $\Sigma(2\ell,2,2,2)$, $\ell\geq 2$ with $c_1(W)=0$. 
 The positive part of symplectic homology of $W$ with coefficients in $\ZZ_2$ is given by
\begin{equation*}
 SH^+_k(W) \cong
 \left\{
 \begin{aligned}
  \ZZ_2 & \qquad \mbox{if } k=2,3 \mbox{ or } k=(2\ell+2)N+j \mbox{ for any }N\in\NN, j\in\{-1,0,1,2\} \\
  (\ZZ_2)^2 & \qquad\mbox{if } k \geq 4, \mbox{ unless } k \mbox{ is as above} \\
  0 & \qquad\mbox{if } k \leq 1.
  \end{aligned}
 \right.
\end{equation*}
\end{thm} 

The case $\ell=1$ is even easier and can be read off directly from \eqref{eq_SC_plus_mb}. For a Liouville filling $W$ of $\Sigma(2,2,2,2)$ with $c_1(W)=0$, we get
\begin{equation*}
 SH^+_k(W) \cong
 \left\{
 \begin{aligned}
  \ZZ_2 & \qquad \mbox{if } k=2 \mbox{ or } k\geq 4 \\
  0 & \qquad\mbox{else.}
  \end{aligned}
 \right.
\end{equation*}

Together with Lemma~\ref{lem_indep_filling}, Theorem~\ref{thm_n3} implies:

\begin{cor}
 The Brieskorn manifolds $\Sigma(2\ell,2,2,2)\cong S^2\times S^3$ with their natural contact structure are pairwise non-contactomorphic. 
\end{cor}

\subsection{A generalization: $\Sigma(\ell p, p, 2, 2)$}

The methods of Section~\ref{sec_perturbation} to \ref{sec_conclusion} can also be applied to $\Sigma(\ell p,p,2,2)$, $p\geq 2$, at least for $p$ even. These manifolds (suggested to me by Otto van Koert) provide further examples for contact manifolds that have the same contact homology (for $p$ fixed), but for which symplectic homology can distinguish the contact structures for different values of $\ell$. Form this point of view, the work above was the special case $p=2$. We sketch the main points for the general case:

\begin{itemize}
 \item Application of Randell's algorithm shows that
 \[
  H_2(\Sigma(\ell p,p,2,2); \ZZ) \cong \ZZ^{p-1}.
 \]
 Using this result and the classification of simply-connected spin $5$-manifolds \cite{Smale}, we get that $\Sigma(\ell p,p,2,2)$ is diffeomorphic to a connected sum of $(p-1)$ copies of $S^2\times S^3$.
 
 \item Analogous to Section~\ref{sec_perturbation}, we use the coordinate change
 \[
  \Sigma(\ell p,p,2,2) \cong \{w \in \CC^4 \mid w_0^{\ell p} + w_1^p + 2 w_2 w_3 = 0\}
 \]
 and perturb the contact form. The resulting simple closed Reeb orbits are
\begin{align} 
 \gamma^+(t) &= \left(0, 0, e^{2it(1+\epsilon)}, 0\right), \qquad 0\leq t\leq\frac{\pi}{1+\epsilon}, \\
 \gamma^-(t) &= \left(0, 0, 0, e^{2it(1-\epsilon)}\right), \qquad 0\leq t\leq\frac{\pi}{1-\epsilon}
\end{align}
 and
\begin{equation}
 \gamma^{0,k}(t) = \left(re^{4it/\ell p}, \zeta^{2k+1} r^\ell e^{4it/p},0, 0\right), \qquad 0\leq t\leq \ell p \frac{\pi}{2},
\end{equation}
 where $r>0$ is the constant satisfying $r^{2\ell}+r^2=1$, $\zeta = e^{\pi i/p}$ is a primitive $2p$-th root of unity and $k=0,\ldots, p-1$. The main difference from Section~\ref{sec_perturbation} is that we get $p$ different simple orbits living in the first two coordinates. 
 The Conley--Zehnder indices of these orbits (and their multiple covers) are similar to \eqref{ustil_ind_1} to \eqref{ustil_ind_3}, namely
\begin{align}
 \mu_\CZ(N\gamma^{0,k}) &= 2N+2N\ell \quad \stackrel{N'\defeq N\ell }{=} \quad 2\frac{N'}{\ell }+2N' \\
 \mu_\CZ(N\gamma^+) &= 2\ceil*{\frac{2N}{\ell p}} +2\ceil*{\frac{2N}{p}} - 2 \\
 \mu_\CZ(N\gamma^-) &= 2\floor*{\frac{2N}{\ell p}} +2\floor*{\frac{2N}{p}} + 2.
\end{align}
At this point, one sees that contact homology cannot distinguish different values of $\ell$. Indeed, checking the indices gives
\[
 CH_k(\Sigma(\ell p,p,2,2);\QQ) \cong 
 \begin{cases}
  \QQ^{p-1} & \qquad \text{for $k=2$}\\
  \QQ^p & \qquad \text{for $k\geq 4$ even} \\
  0 & \qquad \text{else.}
 \end{cases}
\]

\item Each of these orbits gives two generators for $SH_*^+$, corresponding to minimum and maximum of a Morse function on $S^1$. Putting them in a table analogous to Table~\ref{table_generators}, we see that there are $p-1$ generators in degrees $2$ and $3$ and $p$ generators in all higher degrees. Moreover, the generators in degrees $2N(\ell +1)$ and $2N(\ell +1)+1$ (for $N\in \NN$) come exclusively from the orbits $N \gamma^{0,k}$, while in other degrees, they come from $N\gamma^+$ and $N\gamma^-$.

\item Now, assume that $p$ is even and that we use $\ZZ_2$-coefficients. Again, we get an essential ingredient from the full Morse--Bott setup of Section~\ref{sec_MB}. After working through this setup (which now involves three critical submanifolds), one sees that there are only $p-1$ generators in degrees $2N(\ell +1)+j$ for $N\in\NN$ and $j\in\{-1,0,1,2\}$. Hence, there has to be a non-zero differential involving the generators from $N\gamma^{0,k}$.

\item As in Sections \ref{sec_idea} and \ref{sec_transversality}, one can show that the Floer cylinders between orbits $N\gamma^+$ and $N\gamma^-$ come in pairs. Hence, there is no differential between these orbits over $\ZZ_2$-coefficients. (Here, the assumption that $p$ is even is essential, otherwise, there is no $\ZZ_2$-symmetry.) As a consequence, $\rank(SH_k^+(W)) = p$ for $k=4,5,\ldots, 2\ell$ but $\rank(SH_k^+(W)) = p-1$ for $k=2\ell+1, \ldots, 2\ell+4$ (and, by an analog of Lemma~\ref{lem_remaining_degrees}, equality hold in the latter identity, but this is not needed). By Lemma~\ref{lem_indep_filling}, we get

\end{itemize}

\begin{thm} \label{thm_pl}
 For $p$ even, the manifolds $\Sigma(\ell p,p,2,2), \ell\geq 1$ with their canonical contact structures are all diffeomorphic to $\#_{p-1} S^2\times S^3$ and have the same contact homology, yet they are pairwise non-contactomorphic.
\end{thm}

A natural question is whether the same is true for $p$ odd. The obvious thing to try is to apply the same strategy for the $\ZZ_p$ symmetry generated by
\[
 \psi(w_0,w_1,w_2,w_3) = (e^{2\pi i/p} w_0, e^{2\pi i/p} w_1,w_2,w_3)
\]
and $\ZZ_p$ coefficients. On the face of it, everything seems to work fine. However, more checks need to be done, in particular about the orientations of contributing Floer cylinders. This may be subject of future work.

\section{Exotic contact structures on $S^7, S^{11}, S^{15}$} \label{sec_exotic_S7}

The goal of this section is to prove Theorem~\ref{thm_intro_exotic_spheres}. The main focus will be on dimension $7$, where the Brieskorn manifolds $\Sigma(78k+1,13,6,3,3)$ will provide the exotic contact structures. We will first prove that these manifolds are diffeomorphic to $S^7$ and that their contact structure is homotopically trivial. Finally, we can distinguish them via the mean Euler characteristic. For higher dimensions, finding a similar example is difficult (if it exists at all, it has very large exponents). However, the theorem can be proven with the help of connected sums, see Section~\ref{sec_higher_dim}.

\subsection{Diffeomorphism types of topological spheres} \label{sec_diffeo_type}

Recall that a manifold $M$ is called \emph{boundary-parallelizable} if there exists a parallelizable manifold with boundary $W$ such that $\dd(W)=M$.
Let $M,M'$ be boundary-parallelizable homotopy spheres of dimension $4m-1$, $m>1$. Denote by $W, W'$ their parallelizable fillings and by $\sigma(W), \sigma(W')$ the signatures of their intersection products on $H_{2m}(W), H_{2m}(W')$.

By \cite{KerMil}, $M$ is orientation-preserving diffeomorphic to $M'$ if and only if 
\begin{equation}\label{eq_sig_diffeo}
 \sigma(W) \equiv \sigma(W') \mod \sigma_m,
\end{equation}
where $\sigma_m$ is a constant depending only on the dimension. Explicitly, 
\begin{equation}\label{eq_sigma_m}
 \sigma_m = 2^{2m+1}\cdot (2^{2m-1}-1) \cdot \mathrm{numerator}\left(\frac{4B_m}{m}\right),
\end{equation}
where $B_m$ is the $m$-th Bernoulli number, with the convention $B_1=1/6, B_2=1/30,B_3=1/42,B_4=1/30$ and so on.\footnote{Kervaire and Milnor prove \eqref{eq_sigma_m} for $m$ odd, while for $m$ even, it was left open whether there might be another factor of two in some cases. This uncertainty was removed later, see e.g.\ \cite[Theorem~5.2]{Lance} and the references therein.}
In particular, for $m=2$, this formula gives $\sigma_m = 224$.\footnote{This number is easier to understand by noting that the $\sigma(W)$ is divisible by $8$, and the number $\sigma(W)/8 \mod 28$ distinguishes the $28$ smooth structures on $S^7$.}

To apply this result, we need to know the signature of the filling of Brieskorn manifolds. For this, we use \cite[Theorem~3]{Brieskorn}:

\begin{thm}[Brieskorn] \label{thm_signature}
 Assume that $\Sigma = \Sigma(a_0, \ldots, a_n)$ is a homotopy sphere, with $n\geq 4$ even. Denote its filling by $W_a$. Then
 \[
  \sigma(W_a) = \sigma_a^+ - \sigma_a^-,
 \]
 where
 \begin{align}
  \sigma_a^+ &= \#\left\{j=(j_0, \ldots, j_n) \mid 0<j_k<a_k \ \forall k,\ 0<\sum_{k=0}^n \frac{j_k}{a_k} < 1 \mod 2 \right\} \label{eq_a+} \\
  \sigma_a^- &= \#\left\{j=(j_0, \ldots, j_n) \mid 0<j_k<a_k \ \forall k,\ 1<\sum_{k=0}^n \frac{j_k}{a_k} < 2 \mod 2 \right\} \label{eq_a-}.
 \end{align}
\end{thm}
By the condition $0<x<1 \mod 2$ for a real number $x$, we mean that $x$ lies in some interval $(2k,2k+1)$, $k\in\ZZ$, and similarly for $1<x<2 \mod 2$. The numbers $\sigma_a^+$ and $\sigma_a^-$ are precisely the dimensions of the subspaces of $H_n(W_a)$ on which the intersection form is positive and negative, respectively.

We want to apply Theorem~\ref{thm_signature} to $\Sigma(78k+1,13,6,3,3)$. Note that $\Sigma(78k+1,13,6,3,3)$ is a homotopy sphere by Theorem~\ref{thm_topol_sphere}.

\begin{prop} \label{prop_signature}
 The filling $W_k$ of $\Sigma(78k+1,13,6,3,3)$ has signature $\sigma(W_k) = 5824k$, with $\sigma_a^+ = 12272k$ and $\sigma_a^- = 6448k$. In particular, $\Sigma(78k+1,13,6,3,3)$ has the diffeomorphism type of the standard sphere.
\end{prop}

\begin{proof}
 For a tuple $j=(j_0,\ldots, j_4)$ with $0<j_k<a_k$, denote 
 \[
  y_j \defeq \sum_{i=1}^4 \frac{j_i}{a_i} = \frac{j_1}{13} + \frac{j_2}{6} + \frac{j_3}{3} + \frac{j_4}{3},
 \]
 ignoring $j_0$ for the moment. We can write $y_j = \frac{p}{78}$ for some positive integer $p$ (relatively prime to $13$).
 
 The integer $j_0$ can take any value from $1$ to $78k$. For $0\leq n<78$, define
 \[
  I_n \defeq \{nk+1, nk+2, \ldots, (n+1)k\}.
 \]
 The important point of the proof is that for any $j_0\in I_n$, we get the inequality
 \[
  \frac{n}{78} < \frac{j_0}{a_0} = \frac{j_0}{78k+1} <\frac{n+1}{78}.
 \]
 Therefore, if we add $\frac{j_0}{a_0}$ to $y_j=\frac{p}{78}$, the result lies in the same integer interval for all $j_0\in I_n$. 
 It also lies in the same integer interval as $\frac{n+1}{79} +y_j$
 
 For $k=1$, the proposition is just a trivial computation (most easily done by a computer). However, with the above considerations, we can infer the general case $k>1$ from $k=1$. 
 Indeed, we can associate to any tuple $\tilde j=(\tilde j_0,\ldots, \tilde j_4)$ from the ($k=1$)-case (i.e.\ with $0<\tilde j_0<79$) a set of $k$ different tuples $j=(j_0, \ldots, j_4)$ such that
 \[
  \frac{\tilde j_0}{79} + \sum_{i=1}^4 \frac{\tilde j_i}{a_i} \qquad \text{ and } \qquad \frac{j_0}{78k+1} + \sum_{i=1}^4 \frac{j_i}{a_i}
 \]
 lie in the same integer interval. Explicitly, we set 
 \[
  j_i = \tilde j_i \qquad \text{and} \qquad j_0 = (\tilde j_0-1)\cdot k+1, (\tilde j_0-1)\cdot k+2, \ldots, \tilde j_0 \cdot k.
 \]
 This implies that any tuple $\tilde j$ contributing to \eqref{eq_a+} (resp.\ \eqref{eq_a-}) for $k=1$ gives $k$ contributions to \eqref{eq_a+} (resp.\ \eqref{eq_a-}) for $k>1$, and all tuples $j$ are reached from some $\tilde j$ in this way. Thus, $\sigma_a^+$ and $\sigma_a^-$ (and hence $\sigma(W_a)$) both get multiplied by $k$, giving the result. 
\end{proof}

\subsection{Almost contact structure} \label{sec_ac}

\begin{defn} 
 An \emph{almost contact structure} on a manifold $M$ of dimension $2n-1$ is a reduction of the structure group from $SO(2n-1)$ to $U(n-1)\times \id$. Equivalently, if $f: M \to BSO(2n-1)$ denotes the classifying map of the tangent bundle, an almost contact structure is a lift $\bar f: M \to B(U(n-1)\times \id)$, i.e.\ a map $\bar f$ such that the diagram
 \[
 \begin{tikzcd}[column sep=large]
  & B(U(n-1)\times \id) \arrow{d} \\
 M \arrow{ru}{\bar f} \arrow{r}{f} & BSO(2n-1)
 \end{tikzcd}
 \]
 commutes.
\end{defn}
A (cooriented) contact structure $\xi=\ker(\alpha)$ induces an almost contact structure by the splitting $TM = \xi \oplus \langle R_\alpha\rangle$. The almost contact structure of a contact structure is also called its formal homotopy class.

The map $B(U(n-1)\times \id)\rightarrow BSO(2n-1)$ is a fibration with fibers $SO(2n-1)/U(n-1)$. 
Moreover, the inclusion $SO(2n-1)\hookrightarrow SO(2n)$ induces a diffeomorphism of the homogeneous spaces (see e.g.\ \cite[Corollary~3.1.3]{Gray})
\[
 SO(2n-1)/U(n-1) \cong SO(2n)/U(n).
\]
Hence, if $M$ is stably parallelizable, the almost contact structures on $M$ are in one-to-one correspondence with homotopy classes of maps from $M$ to $SO(2n)/U(n)$. In particular, if $\Sigma$ is a (topological) sphere, almost complex structures on $\Sigma$ are classified by $\pi_{2n-1}(SO(2n)/U(n))$, with $0\in \pi_{2n-1}(SO(2n)/U(n))$ corresponding to the trivial almost contact structure. By a classical result from Massey \cite{Massey},
\[
 \pi_{2n-1}(SO(2n)/U(n)) \cong 
 \begin{cases}
  \ZZ \oplus \ZZ_2 & \text{ for $n\equiv 0 \mod 4$} \\
  \ZZ_{(n-1)!} & \text{ for $n\equiv 1 \mod 4$} \\
  \ZZ & \text{ for $n\equiv 2 \mod 4$} \\
  \ZZ_{\frac{(n-1)!}{2}} & \text{ for $n\equiv 3 \mod 4$} \\
 \end{cases}
\]
For Brieskorn manifolds diffeomorphic to standard spheres, Morita \cite{Morita} gives an explicit formula for the almost contact structure in terms of the exponents $a_j$. 
Denote by $\xi_a$ the canonical contact structure of $\Sigma(a)$ and by $ac$ the map sending its underlying almost contact structure to the groups above. Further, abbreviate
\[
 S_m \defeq \frac{2^{2m}(2^{2m-1}-1)B_m}{(2m)!},
\]
where $B_m$ denotes the $m$-th Bernoulli number, with the same convention as in \eqref{eq_sigma_m}. Then, Morita's result states that
\begin{equation} \label{eq_Morita}
 ac(\Sigma(a),\xi_a) = 
 \begin{cases}
  \left(\frac{1}{4S_m} \sigma(W_a)- \frac{1}{2} \mu(a),0\right) & \text{ for $n\equiv 0 \mod 4$} \\
  \frac{1}{2} \mu(a) & \text{ for $n\equiv 1 \mod 4$} \\
  -\frac{1}{4S_m} \sigma(W_a)- \frac{1}{2} \mu(a) & \text{ for $n\equiv 2 \mod 4$} \\
  \frac{1}{2} \mu(a) & \text{ for $n\equiv 3 \mod 4$}
 \end{cases}
\end{equation}
Here, $\mu(a)=\prod_{j=0}^n (a_j-1)$ is the rank of $H_n(W_a)$.
In dimension $7$, we have $m=2$, hence
\[
 ac(\Sigma(a),\xi_a) = \left(\frac{45}{28} \sigma(W_a) - \frac{1}{2} \mu(a), 0\right).
\]
The standard almost contact structure on $S^7$ is represented by $(0,0)$. Hence, we want
\[
 \frac{45}{28} \sigma(W_a) = \frac{1}{2} \mu(a),
\]
or, expressed in the dimensions of the positive and negative eigenspaces of the intersection form (with $\sigma(W_a)=\sigma_a^+ - \sigma_a^-$ and $\mu(a)=\sigma_a^+ + \sigma_a^-$),
\[
 31\sigma_a^+ = 59 \sigma_a^-.
\]
By Proposition~\ref{prop_signature}, this condition is satisfied for $\Sigma(78k+1,13,6,3,3)$. Hence:

\begin{thm} \label{thm_diffeo-ac}
 For any $k\in\NN$, the Brieskorn manifold $\Sigma(78k+1,13,6,3,3)$ is diffeomorphic to $S^7$. Moreover, its canonical contact structure is homotopically standard, i.e.\ its underlying almost contact structure is homotopic to that of $S^7$.
\end{thm}

At this point, one could already use \cite[Theorem~6.1]{Eliash2} to see that the Brieskorn manifolds $\Sigma(78k+1,13,6,3,3)$ give exotic but homotopically standard contact structures on $S^7$. However, it is not yet clear that they are pairwise non-contactomorphic, which we will show in Section~\ref{sec_mean_Euler}.

\subsection{Mean Euler characteristic} \label{sec_mean_Euler}

\subsubsection{General results} \label{sec_mean_Euler_general}

Let $(W, \omega=d\lambda)$ be a Liouville domain with boundary $M=\dd W$. Assume that $M$ is simply-connected and that the first Chern class $c_1(W)$ vanishes on $\pi_2(W)$. Under these conditions, we can associate to $W$ its $S^1$-equivariant symplectic homology $SH_*^{S^1}(W;\QQ)$ with rational coefficients, and in particular its positive part $SH_*^{S^1,+}(W;\QQ)$ (see \cite{BO_gysin}).  We use the latter to define the $i$-th Betti number of the positive $S^1$-equivariant symplectic homology as
\[
 b_i(W) \defeq \dim\left(SH_i^{S^1,+}(W;\QQ)\right).
\]
Now, assume that there exists a chain complex for positive $S^1$-equivariant symplectic homology for which the rank of the chain groups of each degree is uniformly bounded. This chain complex can either come from a contact form with non-degenerate Reeb orbits or from a suitable Morse--Bott setup.  
Then, we can define the \emph{mean Euler characteristic} as
\[
 \chi_m(W) \defeq \frac{1}{2} \left(\liminf_{N\to \infty} \frac{1}{N} \sum_{i=-N}^N (-1)^i b_i(W) + \limsup_{N\to \infty} \frac{1}{N} \sum_{i=-N}^N (-1)^i b_i(W) \right)
\]
In all the cases considered in this paper, the limit actually exists so the formula reduces to
\[
 \chi_m(W) = \lim_{N\to \infty} \frac{1}{N} \sum_{i=-N}^N (-1)^i b_i(W).
\]

By \cite[Corollary~2.2]{FSvK} and with the assumption made above, the mean Euler characteristic depends only on $M$ and its contact structure, i.e.\ it is independent of the filling $W$. Therefore, we will also write $\chi_m(M)$ instead of $\chi_m(W)$.

The next proposition gives an explicit formula for the mean Euler characteristic. Let $(\Sigma, \xi=\ker \alpha)$ be a contact manifold with a Morse--Bott contact form $\alpha$. Assume that the Reeb vector field induces an $S^1$-action with finitely many orbit spaces. Denote the periods, in increasing order, by $T_1<T_2<\cdots <T_k$ and the orbit spaces by $\Sigma_{T_i}$. So $T_k$ is the period of the principal orbit and all $T_i$ divide $T_k$. Define the \emph{frequency}
\[
 \phi_{T_i; T_{i+1}, \ldots, T_\ell } = \# \{a\in\NN \mid aT_i < T_\ell  \text{ and } aT_i\notin T_j\NN \text{ for any } j=i+1, \ldots, \ell\}.
\]
By convention, $\phi_{T_k; \emptyset} = 1$. 

\begin{prop}[\cite{FSvK}, \cite{KvK}] \label{prop_chi_m}
 Let $(\Sigma, \xi=\ker \alpha)$ be a contact manifold as above. Assume the following conditions:
 \begin{itemize}
  \item There exists a exact symplectic filling $(W,d\lambda)$ of $\Sigma$ such that $c_1(W) = 0$ and $\pi_1(W)=0$.
  \item $TW|_\Sigma$ is symplectically trivial.
  \item For any periodic Reeb orbit $\gamma$, the linearized Reeb flow is complex linear in some unitary trivialization of $\xi$ along $\gamma$.
  \item The Robbin--Salamon index of the principal orbit $\mu_P\defeq \mu(\Sigma_{T_k})$ does not vanish.
 \end{itemize}
Then the mean Euler characteristic is given by the formula
\begin{equation} \label{eq_chi_m}
 \chi_m(\Sigma) = \frac{\sum_{i=1}^\ell (-1)^{\mu(\Sigma_{T_i})-\frac{1}{2}\dim(\Sigma_{T_i}/S^1)} \phi_{T_i; T_{i+1}, \ldots, T_\ell } \cdot \chi^{S^1}(\Sigma_{T_i})}{|\mu_P|}.
\end{equation}
\end{prop}

Let us briefly explain how this formula arises. The main idea is to use a Morse--Bott spectral sequence, converging to $SH^{S^1,+}(W;\QQ)$, whose first page is given by
 \begin{equation} \label{eq_spectral_seq}
  E_{p,q}^1 = \bigoplus_{\substack{T \text{ such that } \\ \mu(\Sigma_T) -\frac{1}{2}\dim(\Sigma_T/S^1)=p}} H^{S^1}_q(\Sigma_T; \mcL).
 \end{equation}
 Here, $\mcL$ is a real line bundle, meaning that homology with local coefficients is used. However, the third assumption in Proposition~\ref{prop_chi_m} guarantees that this bundle is trivial, so one has coefficients in $\QQ$.
 Then, adding all the contributions from \eqref{eq_spectral_seq} to $\chi_m$ over one period of the $S^1$-action with the correct signs result in the formula \eqref{eq_chi_m}.

For a Brieskorn manifold $\Sigma(a_0,\ldots, a_n)$, the principal period is $T_k=\lcm(a_i)\frac{\pi}{2}$. Hence, by \eqref{eq_Maslov_index}, the Robbin--Salamon index of a principal orbits is
\begin{equation} \label{eq_principal_Maslov}
 \mu_P = \sum_{j=0}^n \left(\floor*{\frac{\lcm(a_i)}{a_j}}+\ceil*{\frac{\lcm(a_i)}{a_j}}\right) - 2\lcm(a_i) = 2 \lcm(a_i) \cdot \left(\sum_{j=0}^n \frac{1}{a_j}-1 \right)
\end{equation}
So the assumption $\mu_P\neq 0$ is satisfied if $\sum_{j=0}^n \frac{1}{a_j} \neq 1$, while the other assumptions are satisfied for all Brieskorn manifolds.

\subsubsection{Application to $\Sigma(13,11,7,4,3)$} \label{sec_13_11_7_4_3}

Before turning to the main example in Section~\ref{sec_78k}, we briefly show that, if one is willing to use connected sums, there are even easier examples. 
They are based on the formula for the mean Euler characteristic for a connected sum \cite[Theorem~5.19]{KvK}:

\begin{prop} \label{prop_conn_sum}
 Let $\Sigma_1$, $\Sigma_2$ be contact manifolds of dimension $2n-1$ that come along with Liouville fillings for which the mean Euler characteristic is defined. Then
 \begin{equation*} 
 \chi_m(\Sigma_1 \# \Sigma_2) = \chi_m(\Sigma_1) + \chi_m(\Sigma_2) + (-1)^n \frac{1}{2}.
 \end{equation*}
\end{prop}

Here, we will use $\Sigma = \Sigma(13,11,7,4,3)$. Note that it is a homotopy sphere by Theorem~\ref{thm_topol_sphere}. Further, application of Theorem~\ref{thm_signature} shows that the signature of its filling is $1344$. Hence, it is diffeomorphic to $S^7$ and its almost contact structure is zero.

As for the mean Euler characteristic, note that all exponents are pairwise relatively prime. This makes the computation somewhat easier, as \cite[Proposition~4.6]{FSvK} gives a simplified formula for such Brieskorn manifolds. Plugging in the numbers gives
\[
 \chi_m(\Sigma(13,11,7,4,3)) = -\frac{3047}{2546}.
\]
Of course, it can also be worked out directly from Proposition~\ref{prop_chi_m}, with a computation similar to the one we do in Section~\ref{sec_78k}.

As $\chi_m(S^7)=-1/2$ for the standard contact structure, this shows that the contact structure on $\Sigma(13,11,7,4,3)$ is exotic. In order to generate infinitely many exotic contact structures, take the connected sum of $k$ copies of $\Sigma(13,11,7,4,3)$ and use Proposition~\ref{prop_conn_sum} to get
\[
 \chi_m(\#_k \Sigma(13,11,7,4,3)) = -k\cdot \frac{3047}{2546} + (k-1)\cdot \frac{1}{2} = -\frac{1}{2}-k \cdot \frac{1774}{2546},
\]
which is strictly monotone decreasing in $k$. Hence, the manifolds $\#_k \Sigma(13,11,7,4,3)$ are pairwise non-contactomorphic, and we get infinitely many exotic contact structures in $S^7$.

\subsubsection{Application to $\Sigma(78k+1,13,6,3,3)$} \label{sec_78k}

The example $\Sigma(78k+1,13,6,3,3)$ is particularly nice because it does not need the connected sum construction. By Theorem~\ref{thm_diffeo-ac}, we already know that these manifolds are diffeomorphic to $S^7$ and have trivial almost contact structure. Now, we compute their mean Euler characteristic.

First, according to \eqref{eq_principal_Maslov}, the Robbin--Salamon index of the principal orbit is
\[
 \mu_P = 2\cdot \lcm(a_j) \cdot \left(\sum_{j=0}^4 \frac{1}{a_j}-1 \right) = 156 -14 a_0 = 142 - 1092k.
\] 
Computing all the terms appearing in \eqref{eq_chi_m}, we get Table~\ref{tab_chi_m}.

\begin{table}[h]
\begin{center}
\begin{tabular}{l|l|l|l}
 Orbit space & period $/\frac{\pi}{2}$ & $\chi^{S^1}$ & frequency \\
 \hline
 $\Sigma(a_0,13,6,3,3)$ & $78 a_0$ & $4$ & $1$ \\
 $\Sigma(13,6,3,3)$ & $78$ & $3$ & $a_0-1=78k$ \\
 $\Sigma(a_0,6,3,3)$ & $6a_0$ & $3$ & $12$ \\
 $\Sigma(6,3,3)$ & $6$ & $0$ & $12(a_0-1)=12\cdot 78k$ \\
 $\Sigma(a_0,13,3,3)$ & $39 a_0$ & $3$ & $1$ \\
 $\Sigma(13,3,3)$ & $39$ & $2$ & $a_0-1=78k$ \\
 $\Sigma(a_0,3,3)$ & $3a_0$ & $2$ & $12$ \\
 $\Sigma(3,3)$ & $3$ & $3$ & $12(a_0-1)=12\cdot 78k$ \\
 $\Sigma(a_0,13)$ & $13a_0$ & $1$ & $4$ \\
\end{tabular}
\caption{The contributions to $\chi_m(\Sigma(78k+1,13,6,3,3))$}
\label{tab_chi_m}
\end{center}
\end{table}

Hence, we can compute $\chi_m(\Sigma)$ in terms of $k$:

\begin{align*}
 \chi_m(\Sigma) &= -\frac{4+3\cdot 78k+36 + 3+2\cdot 78k+24+3\cdot 12\cdot 78k + 4}{|142-1092k|} \\
 &= \frac{71 + 3198k}{142-1092k}
\end{align*}

By a simple computation, the function
\[
 x\longmapsto  \frac{71+3198x}{142-1092x}
\]
is strictly monotone increasing. Hence, $\chi_m(\Sigma)$ can distinguish the different values of $k$.

\begin{thm} \label{thm_exotic-S7}
 The canonical contact structures on the Brieskorn manifolds $\Sigma(78k+1,13,6,3,3)$ are all different. Hence, in combination with Theorem~\ref{thm_diffeo-ac}, we get infinitely many exotic but homotopically trivial contact structures on $S^7$.
\end{thm}

\subsection{How this example was found} \label{sec_how_to_find}

In the previous sections, the numbers $(78k+1,13,6,3,3)$ (and $(13,11,7,4,3)$ in Section~\ref{sec_13_11_7_4_3}) seemed to appear out of nowhere. In this section, we describe the strategy to find them.

Let $\Sigma=\Sigma(a_0,\ldots, a_4)$ be any Brieskorn manifold with its standard contact structure $\xi$. Denote, as before, its filling by $W_a$, the middle dimension of its homology by $\mu = \rank H_4(W_a) = \prod_{i=0}^4 (a_i-1)$ and its signature by $\sigma$. We are looking for examples that fulfill the following three conditions:
\begin{enumerate}[(i)]
 \item \label{topol_sphere} $\Sigma$ is a topological sphere, i.e.\ $H_{n-1}(\Sigma)=0$. This can be checked by Randell's algorithm.
 \item \label{smooth_sphere} $\Sigma$ has the standard smooth structure. By \eqref{eq_sig_diffeo}, assuming \eqref{topol_sphere} is satisfied, this is the case if and only if
 \begin{equation*}
  \sigma \equiv 0 \mod 224.
 \end{equation*}
 \item \label{homot_standard} $(\Sigma,\xi)$ has the standard almost contact structure. By \eqref{eq_Morita} and assuming \eqref{topol_sphere} and \eqref{smooth_sphere}, this is equivalent to the condition
 \[
  \frac{45}{28}\sigma - \frac{1}{2} \mu = 0.
 \]
\end{enumerate}

To reformulate these conditions, let $\sigma_a^+$ (resp.\ $\sigma_a^-$) denote, as before, the dimension of the positive (resp.\ negative) eigenspace of $H_n(W_a)$. Then $\sigma = \sigma_a^+-\sigma_a^-$ and $\mu=\sigma_a^+ + \sigma_a^-$, so condition \eqref{homot_standard} becomes
\[
 31 \sigma_a^+ = 59 \sigma_a^-.
\]
This gives $\sigma_a^+ = 59k$ and $\sigma_a^- = 31 k$ for some positive integer $k$. Assuming this, condition \eqref{smooth_sphere} is
\[
 \sigma = \sigma_a^+ - \sigma_a^- = 28 k \stackrel{!}{=} 224s
\]
for another positive integer $s$. Hence, $k=8s$. Putting everything together, conditions \eqref{smooth_sphere} and \eqref{homot_standard} are satisfied (under the assumption of \eqref{topol_sphere}) if and only if
\begin{align} \label{eq_a_in_terms_of_s}
 \sigma_a^+ &= 472 s \\
 \sigma_a^- &= 248 s.
\end{align}
In particular, $\mu = \sigma_a^+ + \sigma_a^- = 720s$.

With these preparations, it seems sensible to search for examples with the help of a computer. The algorithm does the following steps:
\begin{itemize}
 \item Iterate over the integer $s$ in some range, e.g.\ for $1\leq s\leq 60$.
 \item Iterate over all tuples $(b_0,\ldots, b_4)$, $b_j\geq 1$ such that $\prod_{i=0}^4 b_i = 720s$.
 \item Each such tuple gives a candidate $\Sigma(a)$ with $a_j = b_j+1$. Compute the signature of its filling with \eqref{eq_a+} and \eqref{eq_a-}.
 \item If \eqref{eq_a_in_terms_of_s} is fulfilled, use Randell's algorithm to check if $\Sigma$ is also a topological sphere. Otherwise, discard it.
\end{itemize}
With this algorithm, the following list of examples was found (values of $s$ without examples are skipped):

\begin{longtable}{l|l l l l}
 $s=4$ & $\Sigma(11,7,5,5,4)$ \\
 \hline
 $s=5$ & $\Sigma(11,11,7,4,3)$ \\
 \hline
 $s=6$ & $\Sigma(13, 11, 7, 4, 3)$ \\
 \hline
 $s=7$ & $\Sigma(11, 10, 9, 8, 2)$ \\
 \hline
 $s=8$ & $\Sigma(17, 16, 5, 4, 3)$ & $\Sigma(21, 13, 5, 4, 3)$ \\
 \hline
 $s=10$ & $\Sigma(26, 13, 5, 4, 3)$ & $\Sigma(41, 6, 5, 4, 4)$ \\
 \hline
 $s=12$ & $\Sigma(25, 11, 7, 7, 2)$ & $\Sigma(28, 11, 9, 3, 3)$ & $\Sigma(46, 7, 5, 5, 3)$ & $\Sigma(37, 11, 5, 4, 3)$ \\
 \hline
 $s=14$ & $\Sigma(22, 17, 7, 6, 2)$ & $\Sigma(25, 13, 8, 6, 2)$ & $\Sigma(29, 16, 7, 3, 3)$ & $\Sigma(31, 13, 8, 3, 3)$ \\
  & $\Sigma(31, 15, 7, 3, 3)$ & $\Sigma(43, 11, 5, 4, 3)$ & $\Sigma(37, 11, 8, 3, 3)$\\
 \hline
 $s=15$ & $\Sigma(25, 16, 7, 6, 2)$ \\
 \hline
 $s=16$ & $\Sigma(33, 13, 7, 6, 2)$ \\
 \hline
 $s=18$ & $\Sigma(37, 13, 7, 6, 2)$ \\
 \hline
 $s=20$ & $\Sigma(21, 17, 16, 4, 2)$ \\
 \hline
 $s=21$ & $\Sigma(43, 11, 10, 5, 2)$ & $\Sigma(43, 19, 6, 3, 3)$ \\
 \hline
 $s=22$ & $\Sigma(25, 23, 11, 4, 2)$ & $\Sigma(45, 13, 7, 6, 2)$ \\
 \hline
 $s=23$ & $\Sigma(31, 24, 7, 5, 2)$ \\
 \hline
 $s=24$ & $\Sigma(31, 25, 7, 5, 2)$ & $\Sigma(31, 17, 13, 4, 2)$ & $\Sigma(97, 7, 6, 4, 3)$ & $\Sigma(91, 9, 5, 4, 3)$ \\
 \hline
 $s=25$ & $\Sigma(31, 21, 11, 4, 2)$ \\
 \hline
 $s=26$ & $\Sigma(79, 13, 6, 3, 3)$ \\
 \hline
 $s=27$ & $\Sigma(37, 16, 13, 4, 2)$ & $\Sigma(37, 19, 11, 4, 2)$ & $\Sigma(46, 19, 7, 5, 2)$ \\
 \hline
 $s=28$ & $\Sigma(71, 9, 7, 7, 2)$ & $\Sigma(64, 11, 9, 5, 2)$ \\
 \hline
 $s=30$ & $\Sigma(41, 19, 11, 4, 2)$ \\
 \hline
 $s=33$ & $\Sigma(41, 23, 10, 4, 2)$ \\
 \hline
 $s=34$ & $\Sigma(35, 31, 9, 4, 2)$ & $\Sigma(52, 17, 11, 4, 2)$ & $\Sigma(103, 11, 7, 3, 3)$ & $\Sigma(86, 17, 4, 4, 3)$ \\
 \hline
 $s=36$ & $\Sigma(37, 31, 9, 4, 2)$ & $\Sigma(91, 17, 4, 4, 3)$ \\
 \hline
 $s=39$ & $\Sigma(79, 16, 7, 5, 2)$ \\
 \hline
 $s=40$ & $\Sigma(101, 17, 4, 4, 3)$ \\
 \hline
 $s=42$ & $\Sigma(113, 16, 4, 4, 3)$ \\
 \hline
 $s=43$ & $\Sigma(44, 37, 6, 5, 2)$ \\
 \hline
 $s=44$ & $\Sigma(49, 34, 6, 5, 2)$ & $\Sigma(89, 16, 7, 5, 2)$ \\
 \hline
 $s=45$ & $\Sigma(136, 11, 7, 3, 3)$ \\
 \hline
 $s=46$ & $\Sigma(93, 16, 7, 5, 2)$ \\
 \hline
 $s=48$ & $\Sigma(97, 16, 7, 5, 2)$ \\
 \hline
 $s=49$ & $\Sigma(148, 11, 7, 3, 3)$ \\
 \hline
 $s=50$ & $\Sigma(121, 13, 6, 6, 2)$ \\
 \hline
 $s=52$ & $\Sigma(157, 13, 6, 3, 3)$ & $\Sigma(131, 10, 9, 5, 2)$ \\
 \hline
 $s=54$ & $\Sigma(73, 28, 6, 5, 2)$ \\
 \hline
 $s=57$ & $\Sigma(91, 20, 9, 4, 2)$ \\
 \hline
 $s=60$ & $\Sigma(91, 31, 5, 3, 3)$ \\
 \hline
\end{longtable}

The example $\Sigma(13,11,7,4,3)$ appears near the top. It was chosen simply as the first example whose exponents are relatively prime.

Unfortunately, this list does not display a simple regular pattern. Therefore, instead of continuing this brute-force method, one can try to find numbers $a_1,a_2,a_3,a_4$ such that, when $a_0\to \infty$, the ratio $\sigma_a^+/\sigma_a^-$ approaches the value $59/31$. In this computation, one can assume that the contribution of $j_0/a_0$ to \eqref{eq_a+} and \eqref{eq_a-} is spread out evenly over an interval of length one.

Thus, with another brute-force search, the numbers $13,6,3,3$ were found quickly. Then, one can check that the values $79,157,235,313$, etc.\ actually work for $a_0$. With this information, trying the tuples $a=(78k+1,13,6,3,3)$ seems like the obvious choice. The remaining work was to verify conditions \eqref{topol_sphere}, \eqref{smooth_sphere} and \eqref{homot_standard}, as was done in Sections \ref{sec_diffeo_type} and \ref{sec_ac}.

The next example that can be found in this way is $\Sigma(504k+1, 36, 7, 4, 2)$. With the same methods, it can be shown that this example also produces an infinite family of exotic contact structures on $S^7$.

\subsection{Further results} \label{sec_further_results}

Having established the existence of infinitely many contact structures in the standard formal homotopy class on $S^7$, one can ask a similar question for other contact manifolds.
In some cases, the answer is just a corollary of Theorem~\ref{thm_exotic-S7}:

\begin{thm} \label{thm_general_mfd_exotic}
 Let $(M,\xi=\ker(\alpha))$ be a contact $7$-manifold that admits a Liouville filling for which the mean Euler characteristic is well-defined. Then, there exist infinitely many contact structures on $M$ in the formal homotopy class of $\xi$.
\end{thm}

\begin{proof}
 Take the connected sum of $M$ with the manifolds from Theorem~\ref{thm_exotic-S7}. These manifolds have trivial almost contact structure, corresponding to the zero element in $\pi_7(SO(8)/U(4))$. Hence, the lift of the classifying map $M\to BSO(7)$ to $U(3)\times \id$ does not change under the connected sum, so the formal homotopy class stays the same. However, the contact structures can be distinguished by the mean Euler characteristic, using Proposition~\ref{prop_conn_sum}.
\end{proof}

A similar theorem holds in dimensions $4m+1$, where the Ustilovsky spheres take the place of the manifolds from Theorem~\ref{thm_exotic-S7}. See e.g.\ \cite{Espina} for the mean Euler characteristic of the Ustilovsky spheres.

\begin{rmk}
There is also a version of the mean Euler characteristic using contact homology. For this purpose, the examples of Section~\ref{sec_13_11_7_4_3} can be useful: All Reeb orbits in $\Sigma(13,11,7,4,3)$ have Conley--Zehnder index $\leq -3$, so cylindrical contact homology is (conjecturally) well-defined. Hence, one can use these manifolds to prove a variant of Theorem~\ref{thm_general_mfd_exotic} in which the assumption of a Liouville-filling is replaced by the assumption that cylindrical contact homology (and its mean Euler characteristic) is well-defined. Besides Brieskorn manifolds, e.g.\ the prequantization bundles from \cite[Example 8.2]{Espina} satisfy this assumption.
\end{rmk}

\subsubsection{Different formal homotopy classes and exotic $7$-spheres}

One may also ask whether there are infinitely many contact structures in other formal homotopy classes on $S^7$. The next proposition gives a partial answer to this question.

\begin{prop} 
 In any almost contact structure of the form $(2k,0)\in\ZZ\oplus\ZZ_2$ on $S^7$, there are infinitely many contact structures.
\end{prop}

\begin{proof}
 We use certain Brieskorn manifolds to construct a manifold diffeomorphic to $S^7$ with almost contact structure $(\pm 2,0)$. Taking connected sums and applying Theorem~\ref{thm_general_mfd_exotic} then finishes the proof.
 
 We choose $M_1=\Sigma(11,9,9,5,3), M_2=\Sigma(13,10,9,3,3)$ and $M_3=\Sigma(167,3,2,2,2)$. It is straightforward to verify that these manifolds are diffeomorphic to $S^7$ and that their almost contact structures are $-40$, $72$ and $194$, respectively. Hence,
 \[
  M_4 \defeq 2 M_1 \# M_2 \cong S^7
 \]
 has almost contact structure $-8$. Further,
 \[
  M_5 \defeq 24 M_4 \# M_3 \cong S^7
 \]
 has almost contact structure $+2$, and
 \[
  M_6 \defeq M_4 \# 3 M_5 \cong S^7
 \]
 has almost contact structure $-2$.
\end{proof}

By contrast, the following lemma implies that the remaining almost contact structures on $S^7$ cannot be realized as connected sums of Brieskorn manifolds diffeomorphic to $S^7$.

\begin{lem}
 Any Brieskorn manifold diffeomorphic to $S^{4m-1}$, $m\geq 2$, has almost contact structure of the form $(2k,0)\in \ZZ\oplus \ZZ_2$ (resp.\ of the form $2k\in\ZZ$ if $m$ is odd).
\end{lem}

\begin{proof}
 By Morita's formula \eqref{eq_Morita}, we have
 \[
 ac(\Sigma(a),\xi_a) =
 \begin{cases}
  \left(\frac{1}{4S_m} \sigma(W_a)- \frac{1}{2} \mu(a),0\right) \in \ZZ\times \ZZ_2 & \text{ if $m$ is even} \\
  -\frac{1}{4S_m} \sigma(W_a)- \frac{1}{2} \mu(a) \in\ZZ & \text{ if $m$ is odd.} \\
 \end{cases}
 \]
 We see immediately that, if $m$ is even, the second factor of the almost contact structure always vanishes. It remains to show that the first factor is an even integer.
 
 By the assumption that $\Sigma(a)$ is diffeomorphic to $S^{4m-1}$, we know from \eqref{eq_sig_diffeo} and \eqref{eq_sigma_m} that $\sigma(W_a)$ is a multiple of $\sigma_m$. We have
 \[
  \frac{\sigma_m}{4S_m} = \frac{\mathrm{numerator}\left(\frac{4B_m}{m}\right)\cdot (2m)!}{B_m} \in 2\ZZ,
 \]
 so $\frac{\sigma(W_a)}{4S_m}$ is certainly an even integer.
 
 As for $\mu(a) = \prod_{j=0}^n (a_j-1)$, we use Theorem~\ref{thm_topol_sphere} to infer its divisibility by $4$. First of all, there exists an exponent, say $a_0$, which is relatively prime to all other exponents. We assume that $a_0$ is odd, since otherwise, all other exponents are odd and $\mu(a)$ is divisible by $2^n$. So we already get a factor of two in $\mu(a)$.
 
 If item \eqref{item_two_isolated_pts} of Theorem~\ref{thm_topol_sphere} applies, we get another factor of two for the same reason, so we are done. So assume that item \eqref{item_condition_B} holds with the set $\{a_1,\ldots, a_r\}$. In particular, $a_1, \ldots, a_r$ are even, while $a_0, a_{r+1}, \ldots, a_n$ are odd. Since $r$ is odd and $n$ is even, we have at least two odd exponents. Hence $\mu(a)$ is divisible by four.
\end{proof}

One might ask whether a result analogous to Theorem~\ref{thm_exotic-S7} holds for exotic $7$-spheres. One problem here is that Morita's calculation of the almost contact structure in \cite{Morita} is only valid for standard smooth spheres. Besides, it is not even clear which almost contact structure should be viewed as standard. Therefore, the best we can do is the following:

\begin{cor}
 On any boundary parallelizable homotopy $7$-sphere $M \in bP_8$, there exists an almost contact structure containing infinitely many contact structures.
\end{cor}

\begin{proof}
 All elements of the group $bP_8$ are represented by $\Sigma(6k-1,3,2,2,2)$ (\cite[p.\ 13]{Brieskorn}). Thus, $M$ is diffeomorphic to a Brieskorn manifold, and we can apply Theorem~\ref{thm_general_mfd_exotic} again.
\end{proof}

\subsubsection{Higher dimensions} \label{sec_higher_dim}

An immediate question is whether an analog of Theorem~\ref{thm_exotic-S7} holds in higher dimensions. We may formulate it like this:

\begin{quote}
 Do there exist infinitely many exotic but homotopically trivial contact structures on $S^{4m-1}$ for $m\geq 2$?
\end{quote}
Note that this is not the case on $S^3$, see \cite{Eliash}.
A similar question for $S^{4m+1}$ was answered affirmatively by Ustilovsky \cite{Ustil}.

In general, it seems hopeless to get analogs of $\Sigma(78k+1,13,6,3,3)$ for general dimensions $4m-1$. The reason is that all terms involving the Bernoulli numbers (in particular $\sigma_m$ and $S_m$) get very complicated.

There is a somewhat simpler approach, making heavy use of connected sums. The strategy is to find Brieskorn manifolds $\Sigma_1$ and $\Sigma_2$ such that: 
\begin{itemize}
 \item Both $\Sigma_1$ and $\Sigma_2$ are diffeomorphic to $S^{4m-1}$.
 \item Viewing the almost contact structure as an element of $\ZZ$ (ignoring the second factor if $m$ is even), we have $ac_1 \defeq ac(\Sigma_1)> 0$ and $ac_2\defeq ac(\Sigma_2) <0$.
 \item $\Sigma \defeq (|ac_2|\Sigma_1) \# (ac_1 \Sigma_2)$ (the connected sum of $|ac_2|$ copies of $\Sigma_1$ with $ac_1$ copies of $\Sigma_2$) has mean Euler characteristic $\chi_m(\Sigma)\neq -\frac{1}{2}$.
\end{itemize}

Then, $\Sigma$ is diffeomorphic to $S^{4m-1}$ with trivial almost contact structure. By taking further connected sums of $\Sigma$ with itself, we get infinitely many values for the mean Euler characteristic, hence infinitely many exotic but homotopically trivial contact structures.

Now, the problem is to find such examples for $\Sigma_1$ and $\Sigma_2$. 
Since we require them to be diffeomorphic to $S^{4m-1}$, their signature should satisfy \eqref{eq_sig_diffeo}. So it should be either zero or very large. Unfortunately, there seem to be no examples with signature zero. It would be interesting to see a conceptual reason for this, possibly from the intersection matrix given in \cite{Pham}.
So the signature needs to have a specific, large value. One way to produce such examples is by mimicking Proposition~\ref{prop_signature}. Thus, we first choose numbers $a_1, \ldots, a_n$ with, say, $a_1$ relatively prime to the rest. Then we set 
\[
 a_0^{(k)} \defeq k\cdot\prod_{i=1}^n a_i + 1 
\]
and $\Sigma^{(k)} \defeq \Sigma(a_0^{(k)}, a_1, \ldots, a_n)$.
(We could also choose $a_0 = k\cdot\prod_{i=1}^n a_i - 1$, which would work similarly.)
With the same proof as for Proposition~\ref{prop_signature}, we get $\sigma(\Sigma^{(k)}) = k\cdot \sigma(\Sigma^{(0)})$.
Hence, once we computed $\sigma(\Sigma^{(0)})$, we can choose $k$ such that $\sigma(\Sigma^{(k)}) \equiv 0$ (e.g.\ $k=\sigma_m$) to ensure diffeomorphicity to $S^{4m-1}$.

For $\Sigma_1$, we can choose $\Sigma_1^{(k)} = \Sigma(6k+1, 3,2,\ldots, 2)$. Its signature is
\[
 \sigma = \sigma(\Sigma(6k+1, 3,2,\ldots, 2)) = k\cdot \sigma(\Sigma(7, 3,2,\ldots, 2)) = (-1)^m 8k.
\]
So we can choose 
\[
 k = \frac{\sigma_m}{8} = 2^{2m-2} \cdot (2^{2m-1}-1) \cdot \mathrm{numerator}\left(\frac{4B_m}{m}\right).
\]
As $\mu=12k=\frac{3}{2}\sigma_m$, we get for the almost contact structure
\begin{align*}
 ac &= (-1)^m \frac{\sigma}{4 S_m} - \frac{1}{2} \mu \\
 &= \left(\frac{1}{4 S_m} - \frac{3}{4}\right) \sigma_m.
\end{align*}
To see that this is positive, we use some estimates for $S_m$. First, a well-known identity for Bernoulli numbers states that
\[
 B_m = \frac{2(2m)!}{(2\pi)^{2m}} \cdot \zeta(2m),
\]
where $\zeta$ is the Riemann zeta function (see e.g.\ \cite[p.\ 286]{Milnor-Stasheff}). As $\zeta(2m)$ converges to $1$ very fast, $B_m\approx 2(2m)!/(2\pi)^{2m}$ is a good approximation. Therefore,
\[
 \frac{1}{4S_m} \approx \frac{(2m)!}{2^{2m+2}(2^{2m-1}-1)} \cdot \frac{(2\pi)^{2m}}{2(2m)!} = \frac{\pi^{2m}}{8(2^{2m-1}-1)}\approx \frac{1}{4} \left(\frac{\pi}{2}\right)^{2m}.
\]
It is not hard to make this estimate precise enough to show that $1/4S_m >3/4$ for all $m>2$. So $\Sigma_1$ does indeed fulfill $ac>0$.

The choice of $\Sigma_2$ is more of a problem. In view of the second condition, it seems reasonable to choose $\Sigma_2^{(k)} = \Sigma(k\cdot d(d+1)+1, d+1, d, \ldots, d)$, where $d$ is sufficiently large. Then we expect that $\mu = k\cdot d^2 \cdot (d-1)^n$ is sufficiently large to make $ac_2$ negative. Another plausible choice might be $\Sigma(2d\cdot k+1, 2, d, \ldots, d)$ for $d\gg 1$ odd. However, the precise value of $\sigma(\Sigma_2^{(k)})$ seems extremely hard to compute. Without such a computation at hand, $ac_2<0$ cannot be known for certain, and even assuming this, we cannot verify the third condition $\chi_m\neq -1/2$, although it looks entirely plausible. In this text, we restrict ourselves to dimensions $11$ and $15$, leaving the general case as a conjecture.

In dimension $11$, it turns out that $d=8$ works. So we take
\[
 \Sigma_2^{(k)} = \Sigma(72\cdot k+1, 9, 8, 8, 8, 8, 8).
\]
A computer calculation gives $\mu^{(k)} = 9680832k$ and $\sigma^{(k)} = -1060560 k$, so $k=496= \sigma_m/16$. This gives the almost contact structure $ac_2 = -396387936$, which is indeed negative.

As for the mean Euler characteristics, it turns out that $\chi_m(\Sigma_1) = -77393/130978 \approx -0.5909$ and $\chi_m(\Sigma_2) = 85520029/193850 \approx 441.1660$. So $\Sigma \defeq (|ac_2|\Sigma_1) \# (ac_1 \Sigma_2)$ has mean Euler characteristic $\chi_m = -3345510952696507/12695042650 \approx -263528.9$, for which we just need that it is not equal to $-1/2$.

In dimension $15$, it turns out that $d=8$ is not enough ($ac_2$ would still be positive), but $d=9$ works. Then, the numbers for 
\[
\Sigma_2^{(k)}=\Sigma(90\cdot k+1, 10,9,9,9,9,9,9,9)
\]
are $\mu^{(k)} = 1698693120k$ and $\sigma^{(k)} = 86754800$, so we can choose $k=4064=\sigma_m/16$, giving $ac_2 = -172412979840 < 0$. Then $\Sigma \defeq (|ac_2|\Sigma_1) \# (ac_1 \Sigma_2)$ has trivial almost contact structure and mean Euler characteristic $\chi_m = 744637007679318226185/6671235576398 \approx 111619054.5$. This finishes the proof of Theorem~\ref{thm_intro_exotic_spheres}.

It can be conjectured that this method can be applied in any dimension. The following consideration from stochastics makes it plausible that $ac_2$ will indeed be negative for $d$ sufficiently large. Let $X_0,\ldots,X_n$ be independent random variables, where $X_i$ is distributed uniformly on the discrete set 
\[
\left\{\frac{1}{a_i}, \ldots, \frac{a_i-1}{a_i}\right\}.
\]
Their sum $S_n = \sum_{i=0}^n X_i$ is a random variable on a discrete set inside $(0,n+1)$, and each outcome gives a contribution to the signature of $\Sigma(a_0,\ldots, a_n)$ as in \eqref{eq_a+}, \eqref{eq_a-}. We can try to estimate $\sigma(W_a)$ with the help of the central limit theorem.

First, all $X_i$ have mean value $1/2$ and standard deviation $\varsigma_i = \sqrt{\frac{a_i-2}{12a_i}}$. For $a_i$ large enough, $\varsigma_i \approx \sqrt{\frac{1}{12}}$ becomes a good approximation. The central limit theorem says that
\[
 \frac{S_n - \frac{n+1}{2}} {\sqrt{\sum_{i=0}^n \varsigma_i^2}} \approx \sqrt{\frac{12}{n+1}} \cdot \left(S_n - \frac{n+1}{2}\right)
\]
will converge in distribution to the standard normal distribution. This means that the cumulative density function can be approximated, for large $n$, by
\begin{equation} \label{eq_approx_by_clt}
 F_n(x) \defeq P({S_n\leq x}) \approx \Phi_{0,1}\left(\frac{12}{n+1} \left(x-\frac{n+1}{2}\right) \right),
\end{equation}
where $\Phi_{0,1}$ is the cumulative density function of the standard normal distribution. A numeric computation shows that, if we use the right hand side to compute the signature as in \eqref{eq_a+}, \eqref{eq_a-}, we get that the quotient $\sigma_a^+/\sigma_a^-$ is very close to one. Hence, it can be expected that $\sigma(W_a)$ is much smaller that $\mu$, so that $ac_2$ will be negative.

Of course, this argument is far from being precise. Most importantly, the approximation of $F_n$ with $\Phi_{0,1}$ is only good for heuristic purposes, as it is never exact for finite $n$. The Berry--Esseen theorem (a quantitative version of the central limit theorem) says that the error in \eqref{eq_approx_by_clt} can be of order at most $n^{-1/2}$. This is not good enough for our purposes, because one would need to do this approximation for all positive integers up to $n$, thereby possibly collecting a total error of order $n\cdot n^{-1/2}=\sqrt{n}$.

Besides, one needs an argument that the mean Euler characteristic of $\Sigma \defeq (|ac_2|\Sigma_1) \# (ac_1 \Sigma_2)$ cannot be $-1/2$. In the examples, it is far away from this value, but of course that requires a proof.


\begin{thebibliography}{bla}
 \bibitem{Ab-Mac} M.\ Abreu, L.\ Macarini, \emph{Contact homology of good toric contact manifolds}, Compositio Math., vol.\ 148, p.\ 304--334, 2012.
 \bibitem{BPS} P.\ Biran, L.\ Polterovich, D.\ Salamon, \emph{Propagation in Hamiltonian dynamics and relative symplectic homology}, Duke Math.\ J.\ 119, no.\ 1 p.\ 65--118, 2003.
 \bibitem{Bourgeois_thesis} F.\ Bourgeois, \emph{A Morse--Bott approach to Contact Homology}, PhD thesis, Stanford University, 2002.
 \bibitem{BO_seq} F.\ Bourgeois, A.\ Oancea, \emph{An exact sequence for contact- and symplectic homology}, Invent.\ Math., 175(3), p.\ 611--680, 2003.
 \bibitem{BO_gysin} F.\ Bourgeois, A.\ Oancea, \emph{The Gysin exact sequence for $S^1$-equivariant symplectic homology}, J.\ Topol.\ Anal.\ 5, no. 5, p.\ 361--407, 2013.
 \bibitem{BO_mb} F.\ Bourgeois, A.\ Oancea, \emph{Symplectic homology, autonomous Hamiltonians, and Morse--Bott moduli spaces}, Duke Math.\ J.\  146, no.\ 1, p.\ 71--174, 2009. 
 \bibitem{Brieskorn} E.\ Brieskorn, \emph{Beispiele zur Differentialtopologie von Singularitäten}, Invent. \ Math.\ 2, p.\ 1--14, 1966.
 \bibitem{Ciel_Fra_Oan} K.\ Cieliebak, U.\ Frauenfelder, A.Oancea, \emph{Rabinowitz Floer homology and symplectic homology}, Annales scientifiques de l'ENS 43, fasc. 6, p.\ 957--1015, 2010.
 \bibitem{Durfee-Kauffman} A.\ Durfee, L.\ Kauffman, \emph{Periodicity of branched cyclic covers}, Math.\ Ann.\ 218, p.\ 157--174, 1975.
 \bibitem{Ding-Geiges} F.\ Ding, H.\ Geiges, \emph{$E_8$-plumbings and exotic contact structures on spheres}, Int.\ Math.\ Res.\ Not., no.\ 71, p.\ 3825--3837, 2004.
 \bibitem{Eliash} Y.\ Eliashberg, \emph{Classification of overtwisted contact structures on 3-manifolds}, Invent.\ Math.\ 98, no.\ 3, p.\ 623–637, 1989.
 \bibitem{Eliash2} Y.\ Eliashberg, \emph{On symplectic manifolds with some contact properties}, J.\ Diff.\ Geom.\ 33, no.\ 1, p.\ 233--238, 1991.
 \bibitem{Espina} J.\ Espina, \emph{On the mean Euler characteristic of contact manifolds}, Internat.\ J.\ Math.\ 25, no.\ 5, 1450046, 2014.
 \bibitem{Fauck} A.\ Fauck, \emph{Rabinowitz-Floer homology on Brieskorn spheres}, Diploma thesis, 2012.
 \bibitem{Fauck_prep} A.\ Fauck, in preparation.
 \bibitem{FHS} A.\ Floer, H.\ Hofer, D.\ Salamon, \emph{Transversality in elliptic Morse theory for the symplectic action}, Duke Math.\ J.\ 80, p.\ 251--292, 1995.
 \bibitem{FSvK} U.\ Frauenfelder, F.\ Schlenk, O.\ van Koert, \emph{Displaceability and the mean Euler characteristic}, Kyoto J.\ Math.\ 52, no.\ 4, p.\ 797--815, 2012.
 \bibitem{Geiges-App} H.\ Geiges, \emph{Applications of contact surgery}, Topology, vol.\ 36, no.\ 6, p.\ 1193--1220, 1997.
 \bibitem{Geiges-book} H.\ Geiges, \emph{An introduction to contact topology}, Cambridge studies in adv.\ math.\, vol.\ 109, Cambridge Univ.\ Press, 2008.
 \bibitem{Gray} J.\ Gray, \emph{Some global properties of contact structures}, Ann.\ of Math.\ 69, no.\ 2, p.\ 421--450, 1959.
 \bibitem{KerMil} M.\ Kervaire, J.\ Milnor, \emph{Groups of homotopy spheres: I} Ann.\ of Math.\ 77, no.\ 3, p.\ 504--537, 1963.
 \bibitem{KvK} M.\ Kwon, O.\ van Koert, \emph{Brieskorn manifolds in contact topology}, preprint, arXiv:1310.0343, 2013.
 \bibitem{Lance} T.\ Lance, \emph{Differentiable structures on manifolds}, Surveys on Surgery Theory, Vol.\ 1, Ann.\ of Math.\ Studies 145, Princeton Univ.\ Press, p.\ 73--104, 2000.
 \bibitem{McLean} M.\ McLean, \emph{Computability and the growth rate of symplectic homology}, preprint, arXiv:1109.4466, 2011.
 \bibitem{Lutz_Meckert} R.\ Lutz, C.\ Meckert, \emph{Structures de contact sur certaines sphères exotiques}, C.\ R.\ Acad.\ Sci.\ Paris Sér. A-B 282, A591–A593, 1976.
 \bibitem{Lerman} E.\ Lerman, \emph{Maximal tori in the contactomorphism groups of circle bundles over Hirzebruch surface}, Math.\ Res.\ Lett.\ 10, no. 1, p.\ 133--144, 2003.
 \bibitem{Massey} W.\ S.\ Massey, \emph{Obstructions to the existence of almost complex structures}, Bull.\ Amer.\ Math.\ Soc. 67, p.\ 559--564, 1961.
 \bibitem{Mc-Sa} D.\ McDuff, D.\ Salamon, \emph{$J$-holomorphic curves and symplectic topology}, AMS Colloquium Publications, vol.\ 52, 2004.
 \bibitem{Milnor} J.\ Milnor, \emph{Singular points of complex hypersurfaces}, Annals of Math.\ Studies 61, Princeton Univ.\ Press, 1968.
 \bibitem{Milnor-Stasheff} J.\ Milnor, J.\ Stasheff, \emph{Characteristic classes}, Annals of Math.\ Studies 76, Princeton Univ.\ Press, 1974.
 \bibitem{Morita} S.\ Morita, \emph{A topological clasification of complex structures on $S^1\times S^{2n-1}$}, Topology 14, p.\ 13-22, 1975.
 \bibitem{Pham} F.\ Pham, \emph{Formules de Picard-Lefschetz généralisées et ramification des intégrales}, Bull.\ Soc.\ math.\ France 93, p.\ 333--367, 1965.
 \bibitem{Randell} R.\ C.\ Randell, \emph{The homology of generalized Brieskorn manifolds}, Topology 14, no.\ 4, p.\ 347--355, 1975. 
 \bibitem{Robbin-Salamon} J.\ Robbin, D.\ Salamon, \emph{The Maslov index for paths}, Topology 32, no.\ 4, p.\ 827--844, 1993.
 \bibitem{Seidel} P.\ Seidel, \emph{A biased view of symplectic cohomology}, Current Developments in Math., Vol.\ 2006, p.\ 211--253, 2008.
 \bibitem{Smale} S.\ Smale, \emph{On the structure of $5$-manifolds}, Ann.\ of Math., vol.\ 75, no.\ 1, p.\ 38--46, 1962.
 \bibitem{Ustil} I.\ Ustilovsky, \emph{Infinitely many contact structures on $S^{4m+1}$}, Internat.\ Math.\ Res.\ Notices, no.\ 14, p.\ 781--791, 1999.
 \bibitem{vanKoert} O.\ van Koert, \emph{Contact homology of Brieskorn manifolds}, Forum Math.\ 20, p.\ 317--339, 2008.
 \bibitem{Wall} C.\ T.\ C.\ Wall, \emph{Classification problems in differential topology, VI-Classification of $(s-1)$-connected $(2s+1)$-manifolds}, Topology 6, p.\ 273--296, 1967.
\end{thebibliography}
\end{document}